\documentclass[a4paper,11pt]{article}
\usepackage[vlined]{algorithm2e}

\SetFuncSty{textsc}
\SetKwFunction{frun}{Run}
\SetKwHangingKw{arun}{\frun}
\SetKwFunction{fset}{Set}
\SetKwHangingKw{aset}{\fset}
\SetKwFunction{fselect}{Select}
\SetKwHangingKw{aselect}{\fselect}
\SetKwFunction{fcompute}{Compute}
\SetKwHangingKw{acompute}{\fcompute}
\SetKwFunction{fsolve}{Solve}
\SetKwHangingKw{asolve}{\fsolve}
\SetKwFunction{festimate}{Estimate}
\SetKwHangingKw{aestimate}{\festimate}
\SetKwFunction{fmark}{Mark}
\SetKwHangingKw{amark}{\fmark}
\SetKwFunction{frefine}{Refine}
\SetKwHangingKw{arefine}{\frefine}
\SetKwFunction{compute}{Compute}
\SetKwFunction{set}{Set}

{\algorithm}%
{\endalgorithm}

\title{Numerical analysis for constrained and unconstrained  $Q $-tensor energies for  liquid crystals}
\author{
Heiko Gimperlein\footnotemark[2]
\and Ruma R. Maity\footnotemark[2]
}

\usepackage{amsmath,amsthm,amssymb,enumerate, amsbsy}
\usepackage[T1]{fontenc}
\usepackage{gensymb}
\usepackage[square,sort&compress,comma,numbers]{natbib}
\usepackage[sc]{mathpazo}
\usepackage{multirow} 
\usepackage{newtxtext,newtxmath}
\usepackage{subfig}
\usepackage{graphicx}
\usepackage{epstopdf}
\usepackage[usenames]{color}
\usepackage[colorlinks=true]{hyperref}
\usepackage{cancel}
\usepackage[framemethod=tikz]{mdframed}
\hypersetup{allcolors=blue}
\usepackage[colorinlistoftodos]{todonotes}

\usepackage{cancel}
\usepackage[margin=3cm]{geometry}
\usepackage{tikz}
\usepackage[titletoc]{appendix}
\usepackage{caption}
\usetikzlibrary{shapes,calc}
\usepackage{verbatim}
\usepackage{mathrsfs}
\usepackage{accents}
\usepackage[utf8]{inputenc}
\usepackage{float}
\restylefloat{table}
\usepackage{multirow,diagbox,tabularx,blindtext,tabulary,tabularht,longtable,blkarray}

\usetikzlibrary{decorations.pathmorphing}
\usetikzlibrary{decorations.pathreplacing}
\usetikzlibrary{positioning}
\usetikzlibrary{shapes}
\usetikzlibrary{arrows}
\usetikzlibrary{patterns}
\usetikzlibrary{fadings}
\usetikzlibrary{plotmarks}
\usetikzlibrary{calc}
\usetikzlibrary{intersections}
\tikzstyle{every picture}+=[font=\footnotesize]
\usepackage{paralist}
\usepackage{empheq}
\usepackage{bbm}
\usepackage{latexsym}           
\usepackage{enumerate}
\usepackage{enumitem}
\usepackage{algorithm2e}

\setlist{noitemsep, topsep=0.8ex, partopsep=0pt
	, leftmargin=3em}
\setlist[1]{labelindent=\parindent}

\newlist{axioms}{enumerate}{1}
\setlist[axioms]{font=\bfseries}

\newlist{alphenum}{enumerate}{1}
\setlist[alphenum]{label=\textbf{(\alph*)}, leftmargin=4em}

\newlist{alphienum}{enumerate}{1}
\setlist[alphienum]{label=\textit{(\alph*)}}

\newlist{romanenum}{enumerate}{1}
\setlist[romanenum]{label=\textit{(\roman*)}}

\newlist{romaninenum}{enumerate*}{1}
\setlist[romaninenum]{label=\textit{(\roman*)}}
\usepackage[vlined]{algorithm2e}
\SetKwIF{If}{ElseIf}{Else}{if}{}{else if}{else}{endif}
\SetKwFor{For}{for}{}{endfor}
\usepackage[noabbrev, capitalise]{cleveref}
\usepackage{tabu}
\crefname{equation}{\unskip}{\unskip}
\creflabelformat{equation}{#2(#1)#3}
\newtheorem{thm}{Theorem}[section]

\newtheorem{lem}[thm]{Lemma}

\theoremstyle{definition}

\theoremstyle{remark}

\newtheorem{rem}[thm]{Remark}

\numberwithin{equation}{section}

\newcommand{\Qvec}{Q}

\newcommand{\Pvec}{P}

\newcommand{\C}{{C}}

\newcommand{\p}{\mathbf{p}}

\newcommand{\abs}[1]{\vert #1 \vert}
\newcommand{\dx}{\,{\rm d}x}

\newcommand{\vdotsop}{%
  \mathinner{\vcenter{
    \baselineskip=1ex
    \hbox{.}\hbox{.}\hbox{.}
  }}
}  

\newcommand{\I}{{\rm I}}

\newcommand{\norm}[1]{{\vert\kern-0.25ex\vert #1 
		\vert\kern-0.25ex \vert}}
\newcommand{\vertiii}[1]{{\vert\kern-0.25ex\vert\kern-0.25ex\vert #1 
		\vert\kern-0.25ex \vert\kern-0.25ex \vert}}
\newcommand{\vertiiih}[1]{{\vert\kern-0.25ex \vert\kern-0.25ex \vert #1 
		\vert\kern-0.25ex \vert\kern-0.25ex \vert}_{h}}
\newcommand{\verti}[1]{{\vert\kern-0.25ex \vert\kern-0.25ex \vert #1 
		\vert\kern-0.25ex \vert\kern-0.25ex \vert}_{1}}	
\newcommand{\vertiiii}[1]{{\vert\kern-0.25ex\vert\kern-0.25ex\vert #1 
		\vert\kern-0.25ex \vert\kern-0.25ex \vert}}

\newcommand{\vertii}[1]{{ \vert\kern-0.25ex \vert\kern-0.25ex \vert #1 
		\vert\kern-0.25ex \vert\kern-0.25ex \vert}_{1}}	
\newcommand{\vertiiiih}[1]{{ \vert\kern-0.25ex \vert\kern-0.25ex \vert #1 
		\vert\kern-0.25ex \vert\kern-0.25ex \vert}_{h}}
\newcommand{\vertiiidg}[1]{{\vert\kern-0.25ex \vert\kern-0.25ex \vert #1 
		\vert\kern-0.25ex \vert\kern-0.25ex\vert}_{\rm dG}}

\newcommand{\vertiiinc}[1]{{\vert\kern-0.25ex\vert\kern-0.25ex\vert #1 
		\vert\kern-0.25ex \vert\kern-0.25ex \vert}_{{\rm pw}}}

\newcommand{\dual}[1]{\langle #1 \rangle}

\newcommand{\vertiiiWP}[1]{{\vert\kern-0.25ex \vert\kern-0.25ex \vert #1 
		\vert\kern-0.25ex \vert\kern-0.25ex\vert}_{\rm P}}
\begin{document}
\maketitle

\renewcommand{\thefootnote}{\fnsymbol{footnote}}
\footnotetext[2]{University of Innsbruck, Innsbruck, Austria (heiko.gimperlein@uibk.ac.at, ruma.maity@uibk.ac.at). }

\begin{abstract} 
This paper introduces a comprehensive finite element approximation framework for three-dimensional Landau-de Gennes  $Q$-tensor energies for nematic liquid crystals, with a particular focus on the anisotropy of the elastic energy and the Ball-Majumdar singular potential. This potential imposes essential physical constraints on the eigenvalues of the $Q$-tensor, ensuring realistic modeling. We address the approximation of regular solutions to nonlinear elliptic partial differential equations with non-homogeneous boundary conditions associated with Landau-de Gennes  energies. The well-posedness of the discrete linearized problem is rigorously demonstrated.  The existence and local uniqueness of the discrete solution is derived using the Newton-Kantorovich theorem. Furthermore, we demonstrate an optimal order convergence rate in the energy norm and discuss the impact of eigenvalue constraints on the {\it a priori} error analysis.

\end{abstract}

\noindent {\bf Keywords:} Conforming finite element method, a priori error analysis, nematic liquid crystals, Landau-de Gennes model, Ball-Majumdar singular potential  
		\medskip
		
		\section{Introduction}\label{introduction}

 Liquid crystals are anisotropic electro-optical materials  with  characteristics of both liquid and of crystalline phases. Widely known for their use in display technology, recent  advances have expanded their applications to include biological sensors, soft robotics, smart windows, and advanced optical devices  \cite{Schwartz_2021,LAGERWALL_2012,Bisoyi_2021,Nesterkina_2024, Christophe_2013}.

Mathematically, the classical Landau-de Gennes (LDG) model encodes the properties of a liquid crystal in a bounded domain $\Omega$ in terms of the $Q$-tensor order parameter, a symmetric, traceless $3\times 3$ matrix depending on $x \in \Omega$. The order parameter $Q$ satisfies a semilinear elliptic boundary value problem involving a nonlinear nonconvex potential function.  

While the numerical analysis for classical LDG  models with a smooth potential has been well understood in recent years, as surveyed below, these smooth models may lead to physically unrealistic predictions. Physically more realistic are singular potentials first introduced by Ball and Majumdar \cite{Ball_Majumdar_2010}. Over the past 5 years finite element methods have been developed for equilibrium configurations  involving such singular potentials in \cite{Schimming_Vinals2020A} for isotropic elastic energies, and in \cite{D3SM01616A,Schimming2022tactoids, Schimming_Vinals2020B, Schimming2021} for anisotropic elastic energies. These methods have been used, in particular, to study topological defects in  nematic liquid crystals, and  tactoids when the nematic and isotropic phases coexist. 

This paper aims to provide a unified {\it a priori} error analysis for such conforming finite element approximations of both smooth and nonsmooth models for nematic liquid crystals. The analysis is new for both the (smooth) Landau-de Gennes  model and the (singular) Ball-Majumdar variant, and it readily allows to include anisotropic elastic energies. 

We restrict to $H^2$ regular solutions. This regularity of the solution is crucial in the context of the singular potential to preserve the eigenvalue constraints. The constraint influences the choice of the asymptotic regime of discretization parameter for the existence and local uniqueness of the discrete solutions, and imposes additional restrictions in comparison with the standard Landau-de Gennes  $Q$-tensor energy. To simplify the presentation, we also restrict to the challenging $3$d case.

For smooth potentials, there is a vast literature of finite element studies of the equilibrium configurations,  bifurcation diagrams,  defect points, and saddle dynamics in various confined geometries with non-homogeneous boundary conditions for the general LDG  $Q$-tensor model \cite{Majumdar_2022_anisotropy, Wang_2017_droplets, Majumdar_2024_LDG_cuboids, Zhang2021_actaNumerica}. In \cite{Gartland1998}, the authors study an abstract approach of the finite element approximation of the non-singular solution branches for the general  $Q$-tensor
 LDG  energy subject to homogeneous boundary conditions. The treatment of non-homogeneous boundary conditions and the influence of parameters on the convergence analysis  for the discontinuous Galerkin finite element approximation of regular solutions that have
 $H^2$-regularity are discussed in \cite{DGFEM} for a reduced two-dimensional (isotropic) LDG  model. These results are extended in \cite{AbstractAMRMNN2023} to include the approximation of exact solutions with minimal regularity,  and  a unified framework for the {\it a priori} and {\it a posteriori} error analysis of different lowest-order finite element methods are developed to approximate the regular solutions of the reduced 2d  LDG model. Recently, \cite{CCADRMNN2024} has derived  lower energy bounds for the nonconvex energy in related Ginzburg-Landau type semilinear problems using Crouzeix-Raviart finite elements to achieve  two-sided energy control of the complicated energy landscapes.

 In spite of the relevance and recent computational advances for singular potentials, like the Ball-Majumdar model \cite{Schimming_Vinals2020A, Schimming_Vinals2020B, Schimming2021}, the numerical analysis of singular model equations remained open. 
\\

This paper is organized as follows. Section \ref{sec:lcmodelsformulation}
reviews the relevant liquid crystal models and the mathematical problem.   
Specifically, Subsection \ref{sec:models} reviews the Landau-De Gennes theory for nematic liquid crystals, including the $Q$-tensor order parameter and physically relevant potentials. 
Section \ref{Problem formulation} describes the full mathematical problem, introducing the elastic anisotropy and a detailed formulation of the singular potential. Existence, physicality and higher regularity of solutions to the Euler-Lagrange equations are discussed. 
Section \ref{FEM-approximation} presents the finite element formulation of the Euler-Lagrange equations, examines the interpolation operator's physicality, and the well-posedness of a  discrete linearized problem. Section \ref{Analysis of nonlinearity}   addresses the treatment of nonlinearity associated with the energy-minimizing equations relevant to the analysis.
 Section \ref{existence-uniqueness} provides the existence, local uniqueness of the discrete solution, and optimal convergence rate.


\subsection*{Notation}
Let $\Tilde{ {S}}$ denotes the set of real symmetric $3\times 3$ matrices, and   $ {S}\subset \Tilde{ {S}}$ denote the set of  symmetric, traceless $3\times 3$ matrices. The set $\Tilde{ {S}}$ is a six-dimensional subspace of $\mathbb{R}^{3\times 3}.$ On $S$ we define the inner product $(\Qvec,\Pvec)_F= \text{tr}(\Qvec^T \Pvec)=\text{tr}(\Qvec \Pvec)$. The associated \emph{Frobenius}  norm is given by $\abs{\Qvec}_{F}=\sqrt{\text{tr}(\Qvec^2)}$. We use the shorthand notation $\Qvec:\Pvec = \text{tr}(\Qvec \Pvec) =\sum_{i,j} \Qvec_{ij} \Pvec_{ij}.$ Let  $\Omega \subset \mathbb{R}^3$ be a three dimensional convex domain. We write $  L^p(\Omega; S) $ for the space of $L^p$-functions in $\Omega$ taking values in $S$, with norm
$\norm{Q}_{L^p(\Omega)}=\left(\int_\Omega \abs{Q}_F^p \dx\right)^{1/p}$, $Q \in  L^p(\Omega; S) . $ More generally, we use $\norm{\bullet}_{L^p(\Omega)}$ also for the $L^p(\Omega)$-norm of any tensor field.
The Sobolev space of $Q$-tensors, $\mathcal{W}^{s,p}(\Omega):={W}^{s,p}(\Omega; S)$, is defined by
\begin{align*}
 {W}^{s,p}(\Omega; S):=\{Q=(Q_{ij}): \Omega \to S \,| \, Q_{ij}\in  W^{s,p}(\Omega; \mathbb{R}) \text{ for all } i,j=1,2,3\}.   
\end{align*}
For the Hilbert spaces $H^s(\Omega; S):={W}^{s,2}(\Omega;S),$ we use the notation $\mathcal{H}^{s}(\Omega).$ 
We define the $\mathcal{H}^1 $ inner product $(Q,P)_1:= \int_\Omega (Q_{ij}P_{ij} + \partial_k Q_{ij} \partial_k P_{ij}  ) \dx,$ and the norm $\norm{Q}_1^2:=(Q,Q)_1$. The  $\mathcal{H}^1(\Omega)$-seminorm is defined as $\abs{P}_1^2:=  \int_\Omega  \partial_k Q_{ij} \partial_k P_{ij}   \dx$.

\section{Liquid crystal models and problem formulation}\label{sec:lcmodelsformulation}
 This section introduces the mathematical models for the general Landau-de Gennes $Q$-tensor energy for nematic liquid crystals, as well as the modification by Ball and Majumdar involving a singular potential.  The existence and physicality of minimizers for both  energy functionals are reviewed, and we allow for elastic anisotropies. Furthermore, the weak formulations of the associated Euler-Lagrange equations are introduced,  higher regularity results for the solutions are explained, and  a linearization of the nonlinear operator around regular solutions is presented.

\subsection{Mathematical models}\label{sec:models}
 
 Nematic liquid crystal molecules are rod-shaped and exhibit a locally preferred orientational ordering pivotal for their  optical and electrical anisotropic properties. In mean field theory, the state of a liquid crystal is described by a probability distribution function $\rho_x(\p)$ of the molecules on the unit sphere $S^2=\{\p \in \mathbb{R}^3| \, \abs{\p}=1 \}$.   For each point $x$ in the material domain $\Omega\subset \mathbb{R}^3,$ and the associated probability density function $\rho_x$  of molecular orientation, the $Q$-tensor order parameter $Q(x)$ at $x \in \Omega$ is the matrix of normalized second order moments of $\rho_x$, defined by \cite{Maire-Saupe-1959, Gennes_Prost_book} 
   \begin{align*} 
   \Qvec(x):=\int_{S^2} \Big(\p \otimes \p - \frac{1}{3} \I \Big) \rho_x(\p) d \p.
\end{align*}
By definition, $\Qvec$ is a symmetric, traceless $3\times 3$ matrix and  its eigenvalues $\lambda_i(\Qvec)$  satisfy the following constraints: 
\begin{align}\label{physical constraint}
    -\frac{1}{3} \leq \lambda_i(\Qvec) \leq \frac{2}{3}, \text{ for } i=1,2,3, \sum_{i=1}^3 \lambda_i(\Qvec)=0.
\end{align}
We refer \eqref{physical constraint} as "physical constraints" and the range $\lambda_i \in ( -\frac{1}{3} ,  \frac{2}{3} )$ as the "physical regime".

 Landau-de Gennes (LDG) theory is the most general continuum theory for nematic liquid crystals. In LDG theory, the state of the nematic phase is described by the macroscopic $Q$-tensor order parameter, where  $Q\in  {S}:=\{Q\in \mathbb{R}^{3\times 3} |\, Q_{ij}= Q_{ji},Q_{ii}=0\}$, but its eigenvalues do not necessarily satisfy the physical constraints \eqref{physical constraint}. This framework abstracts away detailed information about  $\rho$ and intermolecular interactions. 
 
 We consider the LDG energy in the absence of surface energy and external fields, defined by
\begin{align}\label{general_LDG}
    \mathcal{F}[\Qvec]:=  \int_\Omega \mathcal{W}(\nabla Q) \dx + \int_\Omega \psi_B^{LDG}(\Qvec) \dx.
\end{align}
Here, the elastic energy density $\mathcal{W}(\bullet)$ is a function of the spatial derivatives of $Q$, and penalizes  spatial heterogeneity. This energy density is defined explicitly in Section \ref{Problem formulation}. The bulk energy density $\psi_B$ describes the isotropic-nematic phase transition. In LDG theory the bulk energy density is often truncated at the fourth order and defined by a polynomial,
 \begin{align}\label{LDG-bulk-density}
\psi_B^{LDG}(Q):=    -\frac{a}{2} \text{tr}(Q^2)-\frac{b}{3} \text{tr}(Q^3)+\frac{c}{2} (\text{tr}(Q^2))^2.
     \end{align}
    Here, $b,c>0$ are positive material dependent constants. The parameter $a:=\alpha (T^*-T)$ is a rescaled temperature, defined in terms of the characteristic supercooling temperature $T^*$ \cite{Gennes_Prost_book} and a material dependent constant $\alpha>0$. This study focuses on the low-temperature regime where $T<T^\ast$, and therefore $a>0$.

 Majumdar \cite{Majumdar_2010} observed that equilibrium configurations predicted by the LDG theory do not always satisfy the eigenvalue constraints in \eqref{physical constraint} in low-temperature regimes, sometimes resulting in physically unrealistic values. For example, 
 in MBBA liquid crystal material \cite{Majumdar_2010}, with constants $\alpha=0.42 \times 10^3 J/m^3 \, \degree C, b=0.64 \times 10^4 J/m^3 , c=0.35 \times 10^4 J/m^3 , T^* =45  \, \degree C,$ and a nematic-isotropic transition temperature $ T_c=46  \, \degree C,$ the scalar order parameter becomes non-physical within $2 \degree C$ of the transition temperature.
 To address this limitation, Ball and Majumdar \cite{Ball_Majumdar_2010} developed a continuum energy functional based on the Maier-Saupe energy, inspired by Katriel et al.~\cite{Katriel1986}. They introduced a singular potential for the bulk energy, denoted by $f$, which becomes unbounded as the eigenvalues approach  $-1/3$ or $2/3$, thereby enforcing the physicality constraint \eqref{physical constraint}.
 The bulk potential in \cite{Ball_Majumdar_2010}  is defined as 
 \begin{align}\label{BM-bulk-density}
     \psi_B^{BM}(Q):=\begin{cases}
       T f(Q) -\kappa  \text{tr}(Q^2) & \text{if } \lambda_i(\Qvec) \in (-\frac{1}{3},\frac{2}{3}), \, i=1,2,3,\\
        +\infty & \text{otherwise}
     \end{cases}
 \end{align}
where $T$ is  the absolute temperature  and  the  constant $\kappa >0$ is associated to the strength of the intermolecular interactions. The singular bulk potential induces mathematical challenges in both analytical and numerical studies due to the fact that the function $f$ is singular and does not have a closed form in $Q$.  Both the bulk potentials ($ \psi_B^{LDG}(\bullet)$ and $ \psi_B^{BM}(\bullet)$) are qualitatively similar in the sense that both predict a first order nematic-isotropic phase transition, and both are smooth functions in their effective domain. 
 The explicit definition of the map $f$ and its key analytical properties  are detailed in Section \ref{Problem formulation}. For analytical results  of the  $Q$-tensor energy with singular potential, we refer to  \cite{Bauman_Phillips2016, Geng_Tong_2020, Evans2016, Liu_Lin_Xu_2021, Wilkinson2015}.

 Besides the eigenvalue constraint issue, another analytical shortcoming in \eqref{general_LDG} should be noted when seeking a connection between the Oseen-Frank theory and the $Q$-tensor based LDG theory.
 This connection can be established by considering a general four elastic free energy in LDG theory. The most general elastic terms result in an energy that is unbounded from below for certain $Q$ outside the physical regime. This issue of eliminating unphysical behavior of the most general elastic terms  can be resolved using $\psi^{BM}_B(\bullet).$


The local or global minimizers of the constrained or unconstrained $Q$-tensor energies mentioned above correspond to the experimentally observable equilibrium configurations in nematic liquid crystals. In this paper, we focus on an anisotropic elastic energy with  nonzero elastic constants
 $L_1, L_2, L_3$. Under certain conditions on these elastic constants, the Euler-Lagrange equations for $Q$-tensor energies lead to a system of nonlinear elliptic coupled partial differential equations (PDEs) with nonlinearity in the lower order terms. We approximate the regular solutions of these nonlinear PDEs with non-homogeneous Dirichlet boundary conditions (BCs) using the conforming finite element method. The non-homogeneous BCs, or anchoring conditions, play a crucial role in determining the nematic director profiles on surfaces, which is essential for constructing many practical liquid crystal devices \cite{Tsakonas, Caimi2021}.

The analytical behavior of the three-dimensional LDG  $Q$-tensor energy in \eqref{general_LDG}, with isotropic elastic energy ($L_1\neq 0$, $L_2=0=L_3$), and its asymptotic analysis as the elastic constant $L_1\rightarrow 0$ is discussed in \cite{Majumdar_Zarnescu2010}. In the thin film limit, the two-dimensional reduced isotropic LDG model resembles a rescaled version of the celebrated Ginzburg-Landau model \cite{Bethuel1993,Pacard2000}. This model is very well-studied in  literature.
For some analytical results associated with  the anisotropic LDG $Q$-tensor  models, we refer to \cite{Golovaty2019, Golovaty2020}.

Apart from nematic liquid crystals considered here, the general LDG $Q$-tensor model in \eqref{general_LDG} also appears for smectic liquid crystals, coupled with a fourth-order equation for the scalar-valued smectic density variation \cite{Xia2023}, as well as in ferronematic systems \cite{FerroRMAMNN2021,Ferronematics_2D} such as colloidal suspensions of magnetic nanoparticles in a nematic host.

 \subsection{Problem formulation and analysis}\label{Problem formulation}

In this paper, we consider the following anisotropic elastic energy \cite{Ball_Majumdar_2010} for $\Qvec \in \mathcal{H}^1(\Omega),$
\begin{align*}
\mathcal{F}_E(\Qvec)=\int_\Omega  \mathcal{W}(\nabla Q) \dx \text{ with }   \mathcal{W}(\nabla Q):=\frac{1}{2} \big( L_1 \abs{\nabla \Qvec}^2 +L_2 \abs{\nabla \cdot\Qvec}^2 + L_3 (\nabla \Qvec)^T  \vdotsop \nabla \Qvec \big) ,
\end{align*}
where $L_1, L_2, L_3$ 
are material dependent constants,  and $\abs{\nabla \Qvec}^2:= \nabla \Qvec  \vdotsop    \nabla  \Qvec$, $\abs{\nabla \cdot\Qvec}^2:=(\nabla \cdot\Qvec)\cdot (\nabla \cdot \Qvec),$ and for $P,Q \in \mathcal{H}^1(\Omega),$ 
\begin{align*}
   \nabla \Qvec  \vdotsop    \nabla  \Pvec = (\partial_k Q_{ij}) (\partial_k P_{ij}), (\nabla \cdot\Qvec)\cdot (\nabla \cdot\Pvec) := (\partial_j Q_{ij}) (\partial_k P_{ik}),\, (\nabla \Qvec)^T  \vdotsop \nabla \Pvec:= (\partial_j Q_{ik}) (\partial_k P_{ij}) .
\end{align*}
The repeated indices are summed from $1$ to $3$.

 We consider the energy minimization problem,
\begin{align}\label{anisotropy: energy}
    \mathcal{F}[\Qvec]= \mathcal{F}_E(\Qvec)+\mathcal{F}_B(\Qvec),
\end{align}
subject to the non-homogeneous Dirichlet boundary condition $ \Qvec_b\in  H^\frac{3}{2}(\partial\Omega;  {S}).$ It involves the bulk energy, $\mathcal{F}_B(\Qvec):=\mathcal{F}^{LDG}_B(\Qvec):=\int_\Omega \psi_B^{LDG}(\Qvec) \dx$ for the standard LDG $Q$-tensor energy  \eqref{general_LDG}, and $\mathcal{F}_B(\Qvec):=\mathcal{F}^{BM}_B(\Qvec):=\int_\Omega \psi_B^{BM}(\Qvec) \dx$ for the $Q$-tensor energy with singular potential \eqref{BM-bulk-density}.
 \subsubsection*{BM potential and auxiliary  analytical results}
The Ball-Majumdar (BM) singular potential $f$ is defined by
\begin{equation} \label{entropy function}
 f(\Qvec)=  
 \begin{cases}
    \inf_{\rho \in \mathcal{A_\Qvec} }\int_{S^2} \rho(\p) \text{ln}\rho(\p) d\p & \text{if } \lambda_i(\Qvec) \in (-\frac{1}{3},\frac{2}{3}), \, i=1,2,3,\\
        +\infty & \text{otherwise}
    \end{cases}
\end{equation}
where the entropy term is minimized over all probability distributions $\rho$ that have a fixed normalized second moment $Q$, that is,  
\begin{align*}
   \mathcal{A_\Qvec}:=\{ \rho \in  \mathcal{P}| \, \Qvec =\int_{S^2} \big(\p \otimes \p - \frac{1}{3} \I \big) \rho(\p) d \p \} .
\end{align*}
Here $\rho \in \mathcal{P}:=\{\rho \in L^1(S^2; \mathbb{R})|\, \rho \geq 0, \int_{S^2} \rho(\p) d\p =1 \}$ is the equilibrium probability density distribution of the liquid crystal molecules.
For any $\Qvec \in  {S}$, such that $\lambda_i(\Qvec) \in ( -\frac{1}{3} ,  \frac{2}{3} )$ for $i=1,2,3,$ the set $\mathcal{A_\Qvec}$ is non-empty, and  the optimization problem in \eqref{entropy function} has  a unique minimizer in $ \mathcal{A_\Qvec}$  (see \cite[Lemma 1, Theorem 2]{Schimming2021}).

\medskip

\noindent We refer to \cite{Ball_Majumdar_2010,Feireisl2015} for the following analytic properties of the singular potential $f$:
\begin{enumerate}
    \item $f: {S}\rightarrow [-K, +\infty]$ is convex and lower semi-continuous, with $K\geq 0.$
    \item The domain of $f$, i.e.,
    \[\mathcal{D}[f]=\{ \Qvec \in  {S}|\, f( {\Qvec}) <+\infty \} ,  \]
    coincides with the (open, convex, and bounded) subset of $ {S}$ given by
    \[ \mathcal{Q}_{phy}:=\left\{\Qvec \in  {S}|\, \lambda_i(\Qvec )\in (-\frac{1}{3},\frac{2}{3}) \right\}.\]
    \item $f$ is smooth in $\mathcal{D}[f].$
\end{enumerate}
As $Q$ approaches the physical boundary, $\partial \mathcal{Q}_{phy},$ both $f$ and $\frac{\partial f}{\partial Q}$ tend to infinity. Blowup rates of the  singular potential and its gradient as $Q$ approaches $\partial \mathcal{Q}_{phy}$ are discussed in \cite{Ball2017, lu2021blowup}. In particular, there exist  constants $C_l ,C_u$ such that
\begin{align}\label{blow-up-f}
     C_l -    \frac{1}{2} \text{ln}\Big(\lambda_\text{min}(Q)+\frac{1}{3}\Big) \leq f(\Qvec) \leq C_u -   \text{ln}\Big(\lambda_\text{min}(Q)+\frac{1}{3}\Big).
\end{align}
 \subsubsection{Existence of minimizers}
Define the bilinear form $\mathcal{A}(\bullet,\bullet): \mathcal{H}^1(\Omega)\times \mathcal{H}^1(\Omega) \rightarrow \mathbb{R}$  associated to the elastic energy, $\mathcal{F}_E(\bullet),$ 
\begin{align*}
  &  \mathcal{A}(\Qvec,\Pvec):= \int_\Omega \big( L_1\nabla \Qvec  \vdotsop   \nabla \Pvec  +L_2 (\nabla \cdot\Qvec)\cdot (\nabla \cdot\Pvec) 
  + L_3 (\nabla \Qvec)^T  \vdotsop \nabla \Pvec \big)  \dx.
\end{align*}
The boundedness property of $\mathcal{A}(\bullet, \bullet)$ follows readily using the Cauchy-Schwarz inequality, and the coercivity follows under certain conditions on the elastic parameters as discussed below.
 \begin{lem}\label{coercivity-elstic-energy}\cite{Longa_1987}
    (i) The bilinear form $\mathcal{A}(\cdot,\cdot)$ is symmetric and bounded on $\mathcal{H}^1(\Omega)$, that is, $\mathcal{A}(P,R) \leq C_E \abs{P}_1 \abs{R}_1$, where the constant $C_E$ depends on the elasticity constants, $L_1,L_2$ and $L_3.$\\
    (ii) There exist $C_{q_1}, C_{q_2}$ such that $\mathcal{A}(P,P) \geq C_{q_1} \abs{P}^2_1$ for all $P\in \mathcal{H}^1(\Omega),$ and $\mathcal{A}(P,P) \geq C_{q_2} \norm{P}^2_1$ for all $P\in \mathcal{H}^1_0(\Omega),$ where 
    \begin{align}\label{cond-elas-const}
        0<L_1, -L_1< L_3< 2L_1, \text{ and } -\frac{3}{5} L_1 -\frac{1}{10} L_3 <L_2.
    \end{align}
    Here, the constants $C_{q_1}, C_{q_2}$ depend on $\Omega$, and   the elasticity constants, $L_1,L_2$ and $L_3.$
\end{lem}
 Condition \eqref{coercivity-elstic-energy} corresponds to the ellipticity of the problem. It guarantees that the elastic  energy is bounded from below.
\subsubsection*{LDG bulk potential}
 We assume the non-homogeneous Dirichlet boundary condition $ \Qvec_b\in  {H}^\frac{3}{2}(\partial\Omega;S).$ Define the admissible set,  $\mathcal{X}_{LDG}(\Qvec_b ):=\{ \Qvec \in \mathcal{H}^1(\Omega)|\,\, \Qvec =  \Qvec_b \text{ on }{\partial\Omega}\}$. We study the minimization problem 
\[ \min_{\Qvec \in \mathcal{X}_{LDG}(\Qvec_b )} \mathcal{F}_{LDG}[\Qvec], \text{ where } \mathcal{F}_{LDG}[\Qvec]= \mathcal{F}_E(\Qvec)+\mathcal{F}^{LDG}_B(\Qvec). \]
Recall $\mathcal{F}^{LDG}_B(\Qvec):=\int_\Omega \psi_B^{LDG}(Q) \dx$ from \eqref{anisotropy: energy}  and the definition of  $\psi_B^{LDG}(\bullet)$ from \eqref{LDG-bulk-density}. 
The bulk density $\psi_B^{LDG}(\bullet)$ satisfies the growth condition \cite[Corollary 4.4]{Gartland1998}, $ \psi_B^{LDG}(Q) \geq C +C \abs{Q}^2$ for all $Q\in S $ for some constant $C >0.$
Lemma \ref{coercivity-elstic-energy} implies that $ \mathcal{F}_E(\bullet) $ is coercive and convex in $\nabla Q.$ 
Therefore, the coercivity of $\mathcal{F}_{LDG}$  in $\nabla Q$ follows readily from Lemma \ref{coercivity-elstic-energy} and from the above mentioned growth condition of the bulk term. Thus the existence of a minimizer is guaranteed by a straightforward application of  the direct method in the Calculus of Variations.

 \subsubsection*{BM potential}
 For the case of bulk energy associated to the singular potential, we assume that the non-homogeneous Dirichlet boundary condition, $ \Qvec_b\in  {H}^\frac{3}{2}(\partial\Omega;S),$  satisfies an additional condition that $\Qvec_b$ is "strictly physical",   that is,  $\Qvec_b$ is physically realistic in the sense that $-\frac{1}{3}+\widetilde{\epsilon}_0< \lambda_i(\Qvec_b)< \frac{2}{3}-\widetilde{\epsilon}_0$ for $i=1,2,3,$ for a some small parameter $\widetilde{\epsilon}_0 >0,$ or equivalently $\psi_B(\Qvec_b) <\infty.$ For small $\widetilde{\epsilon}>0,$ we define  $\mathcal{Q}_{phy}(\widetilde{\epsilon}):=\{Q\in \mathcal{D}[f]| \,\, -\frac{1}{3}+\widetilde{\epsilon}< \lambda_i(\Qvec)< \frac{2}{3}-\widetilde{\epsilon} \text{ for }i=1,2,3\}$ and  the admissible set  $\mathcal{X}_{BM}(\Qvec_b,\widetilde{\epsilon} ):=\{ \Qvec \in \mathcal{H}^1(\Omega) \cap \mathcal{Q}_{phy}(\widetilde{\epsilon})|\,\, \Qvec =  \Qvec_b \text{ on }{\partial\Omega}\}$.  
 We study the minimization problem 
\[ \min_{\Qvec \in \mathcal{X}_{BM}(\Qvec_b,\widetilde{\epsilon} )} \mathcal{F}_{BM}[\Qvec]:= \mathcal{F}_E(\Qvec)+\mathcal{F}^{BM}_B(\Qvec). \]
Recall $\mathcal{F}^{BM}_B(\Qvec)=\int_\Omega \psi_B^{BM}(Q) \dx  $ from \eqref{anisotropy: energy}, and the definition of  $ \psi_B^{BM}(\bullet) $ from  \eqref{BM-bulk-density}.
For any $\Qvec  \in \mathcal{X}_{BM}(\Qvec_b,\widetilde{\epsilon} ),  $ Lemma \ref{coercivity-elstic-energy}  and the lower bound of $f$ in \eqref{blow-up-f} imply the existence of minimizers.

  \subsubsection*{Physicality of minimizers}
For $\Tilde{Q},$  a minimizer of $\mathcal{F}_{BM}[\cdot]$, the singular potential function in \eqref{entropy function} $f(\Tilde{Q}) < \infty$ almost everywhere. Therefore, the eigenvalues of $\Tilde{Q}$ lies within the physical regime almost everywhere. In fact, for the one constant  elastic free energy ($L_1>0, L_2=L_3=0$) given by
\begin{align}\label{one-constant-elastic-energy}
    \mathcal{F}_{BM}[\Qvec]:=\int_\Omega L_1 \abs{\nabla \Qvec}^2 \dx+ \int_\Omega \psi_B(\Qvec) \dx,
\end{align} 
the authors in \cite{Ball_Majumdar_2010} establish the "strict physicality" of the global minimizers of \eqref{one-constant-elastic-energy} via a maximum principle:
\begin{thm}\cite[Theorem 6]{Ball_Majumdar_2010} \label{physical regime of global minimizer}
    Let $\Qvec^*$ be a global minimizer of $ \mathcal{F}_{BM}[\bullet]$ subject to the fixed boundary condition, $\Qvec=\Qvec_b$ on $\partial \Omega, $ where the boundary condition $\Qvec_b$ is physically realistic in the sense that $-\frac{1}{3}+\widetilde{\epsilon}_0\leq  \lambda_i(\Qvec_b) \leq \frac{2}{3}-\widetilde{\epsilon}_0$ for $i=1,2,3,$ for some $\widetilde{\epsilon}_0 >0,$ or equivalently $\psi_B(\Qvec_b)$ is bounded. Then the minimizer $\Qvec^*$ is strictly physical, that is,  $\Qvec^* \in \mathcal{X}_{BM}(\Qvec_b,\widetilde{\epsilon} )$ for some $\widetilde{\epsilon} >0$.
\end{thm}
 \noindent Corresponding results for local minimizers are discussed in \cite{Bauman_Phillips2016, Evans2016}.

\begin{rem}[Limitations for elastic anisotropy]
The strict physicality (eigenvalue constraint) of the local/global minimizers for the case of elastic anisotropy remains an open problem.  In fact, in \cite{Ball2017}, J.~M.~Ball suggests that one might expect the strict physicality of local minimizers to be true in general, because otherwise the integrand will be unbounded in a neighbourhood of some point of $\Omega$. 
\end{rem}
Throughout this paper, we assume the "strict physicality" of the local/global minimizers of $\mathcal{F}_{BM}[\bullet].$

\subsubsection{Weak formulation and higher regularity}\label{higher regularity}
The Euler-Lagrange equation associated to the energy minimization problem \eqref{anisotropy: energy}  is given by $D\mathcal{F}[\Qvec]=0,$ that is the PDE, 
\begin{align}\label{nonlinear pde}
 \dual{N(\Qvec),\Pvec} :=  \mathcal{A}(\Qvec,\Pvec)+\mathcal{B}(\Qvec,\Pvec)
 =0 \text{ for all } \Pvec \in  \mathcal{H}^1_0(\Omega).
\end{align}
For (a) the LDG bulk potential $\mathcal{B}(\bullet,\bullet):= \mathcal{B}_{LDG}(\bullet,\bullet)=D\mathcal{F}^{LDG}_B(\bullet)$, where
\begin{align*}
     \mathcal{B}_{LDG}(\Qvec,\Pvec):= \int_\Omega  \Big(-a \text{tr}(QP) -  b \text{tr}(Q^2  P)+ c \abs{Q}_{F}^2 \text{tr}(QP) \Big) \dx,
\end{align*}
and for (b) the BM potential, $\mathcal{B}(\bullet,\bullet):= \mathcal{B}_{BM}(\bullet,\bullet)=D\mathcal{F}^{BM}_B(\bullet)$, where 
\begin{align*}
     \mathcal{B}_{BM}(\Qvec,\Pvec):=   \frac{T}{2}\int_\Omega \frac{\partial f (\Qvec)}{\partial \Qvec }: \Pvec \dx -\kappa\int_\Omega   \Qvec:\Pvec \dx  .
\end{align*}

In the following we use that under the ellipticity assumption \eqref{cond-elas-const} the linear problem   $\mathcal{A}(\Qvec,\Pvec) = (G ,\Pvec)_{L^2(\Omega)}$ admits a solution $\Qvec \in \mathcal{H}^2(\Omega)$ whenever $G \in L^2(\Omega)$. This satisfied in  bounded domains $\Omega \subset \mathbb{R}^3$ with a $C^2$ boundary by Section 4.3.2 in \cite{giaquinta2013introduction} and for the  specific bilinear form $\mathcal{A}$ in open, bounded and convex domains $\Omega \subset \mathbb{R}^3$ by Section 6 in \cite{Gartland1998}.

  \subsubsection*{Higher regularity }
(a) (LDG bulk potential). The bulk density $\psi_B^{LDG}(\bullet)$ in \eqref{LDG-bulk-density} is continuously differentiable in $S$, and satisfies the  growth condition $D\psi_B^{LDG}(Q) \leq C(1+\abs{Q}^3)$ for all $Q\in S,$ and for some $C>0.$ This and the Sobolev embedding $ \mathcal{H}^1(\Omega) \hookrightarrow L^6(\Omega;S)$ for $\Omega \subset \mathbb{R}^3$, imply that $D\mathcal{F}^{LDG}_B(\Qvec) \in \mathcal{L}(L^2(\Omega,  {S}),\mathbb{R}).$ Therefore, for a given non-homogeneous Dirichlet boundary condition $\Qvec_b\in  H^\frac{3}{2}(\partial\Omega;  {S})$, the elliptic regularity of the linear problem in $\Omega \subset \mathbb{R}^3$ and a bootstrapping argument yield that the solution  $Q\in \mathcal{H}^1(\Omega)$ of \eqref{nonlinear pde} (associated to $\mathcal{B}(\bullet,\bullet)= \mathcal{B}_{LDG}(\bullet,\bullet)$) is of higher regularity, $Q\in \mathcal{H}^2(\Omega)$, under the condition \eqref{cond-elas-const} on the elastic parameters $L_1, L_2, L_3$.

\medskip

\noindent (b) (BM potential). 
In the case of the singular potential and strictly physical boundary data $ \Qvec_b\in  H^\frac{3}{2}(\partial\Omega;  {S}) \cap \mathcal{D}[f]$ 
 with $-\frac{1}{3}+\widetilde{\epsilon}_0< \lambda_i(\Qvec_b)< \frac{2}{3}-\widetilde{\epsilon}_0$ for $i=1,2,3,$  the solution $\Qvec$ of \eqref{nonlinear pde} belongs to  $\mathcal{X}_{BM}(\Qvec_b,\widetilde{\epsilon} )$ for some $\widetilde{\epsilon} >0.$ For the one constant elastic energy in \eqref{one-constant-elastic-energy}, 
 Theorem \ref{physical regime of global minimizer} justifies the existence of such  $\widetilde{\epsilon} >0.$ 
 Since $f$ is $C^\infty$ in $\mathcal{D}[f]$ and $ \mathcal{H}^1(\Omega) \hookrightarrow L^2(\Omega;S)$, the Hölder's  inequality implies that $D\mathcal{F}^{BM}_B(Q) \in \mathcal{L}(L^2(\Omega,  {S}), \mathbb{R} ).$ This and the elliptic regularity of the linear problem in $\Omega$  imply that the solution $\Qvec$ of \eqref{nonlinear pde} belongs to $\mathcal{H}^2(\Omega).$



\medskip
\subsection*{Linearization around regular solutions}
In this paper we approximate $\mathcal{H}^2(\Omega)$-regular solutions of \eqref{nonlinear pde}. 
For such $\Qvec$,   the nonlinear map $N$ is differentiable at $\Qvec$,  and the Frech\'et derivative $DN(\Qvec)$ is isomorphism of $\mathcal{H}_0^1(\Omega)$ onto $(\mathcal{H}_0^1(\Omega))^*.$ That is,
\begin{align}\label{continuous-inf-sup}
    0<\beta :=
     \inf_{\substack{R\in \mathcal{H}^1_0(\Omega) \\  \norm{R}_1=1}} \sup_{\substack{ {P}\in \mathcal{H}^1_0(\Omega) \\  \norm{ {P}}_1=1}} (\mathcal{A}(R,\Pvec)+ \mathcal{B}_L(R,\Pvec)),
\end{align}
where 
$\mathcal{B}_L(R,\Pvec)$ is the linearization of  $\mathcal{B}(\bullet)$ with respect to $Q$: for (a) the LDG bulk potential, $\mathcal{B}_L(R,P):=D\mathcal{B}_{LDG}(Q; R,P)$ with
\begin{align}\label{LDG-linearization}
   D\mathcal{B}_{LDG}(Q; R,P):=\int_\Omega \big(-a \text{tr}(RP)  -  2b \text{tr}(QRP) + c (\text{tr}(Q^2) \text{tr}(RP) +2\text{tr}(QR) \text{tr}(QP)) \big)\dx.
\end{align}
and for  (b) the BM potential, $\mathcal{B}_L(R,P):=D\mathcal{B}_{BM}(Q; R,P)$ with
\begin{align}\label{BM-linearization}
    &D\mathcal{B}_{BM}(Q; R,P):= \frac{T}{2} \int_\Omega R:  \frac{\partial^2  f(\Qvec)}{\partial \Qvec^2}:\Pvec  \dx -\kappa\int_\Omega   R:\Pvec  \dx  ,  \\&
    \int_\Omega R:  \frac{\partial^2  f(\Qvec)}{\partial^2 \Qvec}:\Pvec  \dx = \int_\Omega \sum_{i,j,k,l} R_{ij} \frac{\partial^2  f (\Qvec)}{\partial \Qvec_{ij} \Qvec_{kl} } \Pvec_{kl}  \dx. \notag
\end{align}
For the BM potential, we  look for $\mathcal{H}^2(\Omega)$-regular solutions in the admissible set $\mathcal{X}_{BM}(\Qvec_b,\widetilde{\epsilon} ).$
\section{Finite element approximation}\label{FEM-approximation}
This section introduces the discrete formulation of the PDE \eqref{nonlinear pde}, and discusses the well-posedness of a discrete linearized system of equations.

Let $\mathcal{T}$ be a shape regular triangulation  of $\Omega \subset \mathbb{R}^3$ into  closed 3-simplices in the sense of \cite{ciarlet}. 
Let  $P_0(\mathcal{T}):=\{p\in L^\infty(\Omega)|\,p|_T \in P_k(T) \text{  for all } T\in \mathcal{T}  \}$ denote the space of piecewise polynomials,  where  $P_k(T)$ be the space of polynomials on $T$ of degree at most $k\in \mathbb{N}_0.$ 
	The maximal mesh-size in the triangulation $\mathcal{T}$ is denoted by $h:=\max h_{\mathcal{T} }$, where $h_{\mathcal{T} }\in P_0(\mathcal{T})$ with $h_{\mathcal{T}}|_T:= h_T $, where  $h_T:=\text{diam}(T)$ denotes the diameter of each simplex $T\in \mathcal{T}.$ 
 Let $\mathbb{T}(\delta)$ denote the non-empty subset of all such triangulations $\mathcal{T} $ with maximal mesh size smaller than or equal to a positive constant $\delta.$ 
 
Let $\{E_i\}_{i=1}^5$  be an orthonormal basis for $ {S}$. 
Then  any $\Qvec \in  {S}$ has a unique representation $\Qvec=\sum_{i=1}^5q_i E_i$ with $q_i=\text{tr}(\Qvec E_i).$ 
 Define $X_\C:=\{v\in C^0(\Omega; \mathbb{R})|\, v|_T\in P_1(T) \text{ for all }T\in \mathcal{T}\}.$ 
The discrete spaces for the approximation of matrices in $ {S}$ are defined as
\begin{align*}
   &  {S}_\C:=\left\{\Pvec \in C^0(\Omega;  {S})|\, \Pvec=\sum_{i=1}^5 p_{i,\C} E_i, \, p_{i,\C} \in  X_\C, 1\leq i \leq 5  \right\} \subset  {S}, \text{ and }  {S}_\C^0:=\left\{\Pvec \in  {S}_\C |\, \Pvec|_{\partial\Omega
    }=\mathbf{0}  \right\}.
\end{align*}

The discrete problem seeks $Q_\C \in  {S}_\C$ with boundary data $Q_\C|_{\partial \Omega}:=\I_\C Q_b$ obtained by conforming interpolation, such that
\begin{align}\label{discrete problem}
   \dual{N(\Qvec_\C),\Pvec_\C} := \mathcal{A}(Q_\C, P_\C)+\mathcal{B}(Q_\C, P_\C) = 0 \text{ for all }P_\C \in   {S}_\C^0.
\end{align}

The next lemma addresses the physicality of the conforming interpolation operator crucial for the analysis.
\begin{lem}[interpolation estimates]\label{Interpolation estimateConforming}
	For any $\Xi \!\in\! \mathcal{H}^2(\Omega)$ there  exists 
  ${\I_\C \Xi } \in   {S}_\C$  such that  (i) $	\norm{\Xi-{\rm{I}}_{\C}\Xi }_{\mathcal{H}^l(T)}  \leq C_I  h_T^{{2}-l}  \norm{\Xi}_{\mathcal{H}^{2}(T)}	$ for all $T \in \mathcal{T}$ with $l=0,1.$ (ii) Furthermore, for  $\Qvec \in \mathcal{Q}_{phy}(\widetilde{\epsilon}),$ there exists $\delta_0>0$  such that for any $\mathcal{T}\in \mathbb{T}(\delta)$ with $0<\delta \leq \delta_0$, the interpolation ${\I_\C \Qvec}$ belongs to the physical regime $\mathcal{Q}_{phy}(\widetilde{\epsilon}_1)$ for some $\widetilde{\epsilon}_1>0$.  
\end{lem}
 \begin{proof}[Proof of (i)]
     For any scalar variable $\xi \!\in\! {H}^{2}(\Omega; \mathbb{R}) $, the existence of the conforming interpolant ${\I_\C \xi} \in P_1(\mathcal{T}) \cap {H}^1(\Omega; \mathbb{R}) $  with the interpolation estimates in (i) is well-known in the literature, and we refer to  \cite{brenner,Ern_Guermond2004_bookFEM}.  Let $\Xi:=\sum_{i=1}^5 \xi_i E_i $.
  It is trivial to see   from  the vector representation $\Xi:=(\xi_1, \xi_2, \xi_3, \xi_4, \xi_5)$ that
     the conforming interpolation ${\I_\C \Xi}:=(\I_\C \xi_1,\I_\C \xi_2,\I_\C \xi_3,\I_\C \xi_4, \I_\C \xi_5) \in (P_1(\mathcal{T}) \cap {H}^{1}(\Omega; \mathbb{R}) )^5$  satisfies the property (i).
 \end{proof}
  \begin{proof}[Proof of (ii)]
For  $\Qvec \in \mathbb{R}^{3 \times 3},$ Weyl's inequality \cite[Section III.2]{1997_Bhatia_book} for the eigenvalues of a symmetric matrix provide
\begin{align*}
   \sup_j \inf_i \abs{ \lambda_j(\Qvec(x))- \lambda_i(\I_\C \Qvec(x))} \leq  \abs{\Qvec(x) - \I_\C \Qvec(x)}_{F} \text{ for all }x\in \Omega, \color{blue}\text{ and }i,j=1, 2,  3.
\end{align*}
This and the interpolation estimate $\norm{\Qvec-{\rm{I}}_{\C}\Qvec}_{L^\infty(\Omega;S) } \leq C_I  h^{1/2}  \norm{\Qvec}_{\mathcal{H}^{2}(\Omega)}	$ leads to $\abs{ \lambda_j(\Qvec(x))- \lambda_i(\I_\C \Qvec(x))} \leq \epsilon_1,$
where $0<\epsilon_1 =C_I h^{1/2}  \norm{\Qvec}_{\mathcal{H}^{2}(\Omega)}.$ 
Since $\Qvec \in \mathcal{Q}_{phy}(\widetilde{\epsilon}),$ 
we obtain $-1/3 +(\widetilde{\epsilon} - \epsilon_1) \leq \lambda_i(\I_\C \Qvec) \leq 2/3 -(\widetilde{\epsilon} - \epsilon_1) $ for $i=1,2,3.$   
Therefore, for $\delta_0:=\widetilde{\epsilon}^2/ C_I^2  \norm{\Qvec}_{\mathcal{H}^{2}(\Omega)}^2$ and  any $\mathcal{T} \in \mathbb{T}(\delta)$ with $0< \delta < \delta_0$, the eigenvalues $\lambda_i(\I_\C \Qvec) \in  ( -\frac{1}{3}+\widetilde{\epsilon}_1,  \frac{2}{3}-\widetilde{\epsilon}_1 )$ for $i=1,2,3$ with $\widetilde{\epsilon}_1:= \widetilde{\epsilon} - \epsilon_1.$ 
\end{proof}
\subsection{Well-posedness of the discrete linear problem}
\begin{thm}\label{thm: dis-inf-sup} For suitably restricted values of  the elastic constants $L_1, L_2, L_3$ as $  0<L_1, -L_1< L_3< 2L_1, \text{ and } -\frac{3}{5} L_1 -\frac{1}{10} L_3 <L_2,$
there exists some $\delta_1 >0$  such that, any $\mathcal{T} \in \mathbb{T}(\delta)$ with $0< \delta \leq \delta_1$, and a regular solution $\Qvec \in \mathcal{H}^2(\Omega)$ of \eqref{nonlinear pde} satisfy
    \begin{align*}
   \beta_1<\beta_h :=
     \min_{\substack{R_\C \in  {S}_\C^0 \\  \norm{R_\C}_1=1}} \max_{\substack{P_\C\in  {S}_\C^0 \\  \norm{P_\C}_1=1}} (\mathcal{A}(R_C,\Pvec_C)+ \mathcal{B}_L(R_C,\Pvec_C)),
\end{align*}
where $\beta_1:= (\beta  C_{q_1} / C_E) - (C_E + C_{q_1}) C_I C_{\rm reg} C_Q  \delta_1$ with $$C_Q:=\begin{cases}
(a+2b\norm{Q}_{L^\infty(\Omega)}+3c \norm{Q}_{L^\infty(\Omega)}^2) & \text{for LDG bulk potential},\\
    \Big(\frac{T}{2}  \Big|\Big|\frac{\partial^2  f (\Qvec)}{\partial\Qvec^2}\Big|\Big|_{L^\infty(\Omega)}
    +\kappa \Big) & \text{for BM bulk potential}.
\end{cases} $$ Recall $\beta $ from \eqref{continuous-inf-sup}, $C_E, C_{q_1}  $ from Lemma \ref{coercivity-elstic-energy}, $C_I$ from Lemma \ref{Interpolation estimateConforming}, and  $ C_{\rm reg}$ is a constant related to the $\mathcal{H}^2$ regularity of the solution.
\end{thm}
\begin{proof}
    For $R_\C \in  {S}_\C^0$ with $\norm{R_\C}_1=1,$ let $\Xi \in \mathcal{H}^1_0(\Omega)$ solve the following system:
    \begin{align}\label{a linearized problem}
   \mathcal{A}(\Xi, P)=\mathcal{B}_L(R_\C, P) \text{ for all } P \in \mathcal{H}^1_0(\Omega).
\end{align}
\noindent Recall $\mathcal{B}_L(\bullet,\bullet)$ from \eqref{LDG-linearization} for (a), the LDG bulk potential,  
\begin{align*}
    \mathcal{B}_L(R_C,P):=\int_\Omega \big(-a \text{tr}(R_C P)  -  2b (\text{tr}(QR_C P) )+ c (\text{tr}(Q^2) \text{tr}(R_C P) +2\text{tr}(QR_C) \text{tr}(QP)) \big)\dx.
\end{align*}
The properties of the trace and Cauchy-Schwarz   inequality, and the Sobolev embedding result $\mathcal{H}^2(\Omega) \hookrightarrow L^\infty(\Omega;S) $ lead to 
\begin{align*}
   & \int_\Omega \text{tr}(R_C P) \dx \leq  \int_\Omega \abs{R_C}_{F} \abs{P}_{F} \dx \leq \norm{R_C}_{L^2(\Omega)} 
   \norm{P}_{L^2(\Omega)},\\&
     \int_\Omega \text{tr}(QR_C P) \dx \leq   \int_\Omega  \abs{Q}_{F} \abs{R_C}_{F} \abs{P}_{F} \dx \leq \norm{Q}_{L^\infty(\Omega)} \norm{R_C}_{L^2(\Omega)} \norm{P}_{L^2(\Omega)}.
\end{align*}
Similarly, we obtain 
\begin{align*}
  &\int_\Omega \text{tr}(Q^2) \text{tr}(R_C P) \dx +2   \int_\Omega \text{tr}(QR_C) \text{tr}(QP) \dx \\&\leq 3\int_\Omega \abs{Q}_{F}^2 \abs{R_C}_{F}\abs{P}_{F} \dx \leq 3\norm{Q}^2_{L^\infty(\Omega)}  \norm{R_C}_{L^2(\Omega)} \norm{P}_{L^2(\Omega)}.
\end{align*}
This implies that given $R_\C \in  {S}_\C^0$ with $\norm{R_\C}_1=1,$ the  right hand side operator in \eqref{a linearized problem} for the LDG bulk potential satisfies   $\mathcal{B}_L(R_\C, \bullet) =D\mathcal{B}_{LDG}(Q;R_\C, \bullet)  \in \mathcal{L}(L^2(\Omega,  {S}),\mathbb{R}),$ and is  bounded by
\begin{align}\label{dis-infsup-fbdd-LDG}
  \abs{D\mathcal{B}_{LDG}(Q;R_\C, \bullet)}  \leq a+2b\norm{Q}_{L^\infty(\Omega)}+3c \norm{Q}_{L^\infty(\Omega)}^2 .
\end{align}
Next we proceed with a similar  estimate for the case of the BM potential, (b). Recall that $\mathcal{B}_L(\bullet,\bullet)$ from \eqref{BM-linearization} for the BM potential is given by 
 $$D\mathcal{B}_{BM}(Q; R_\C,P) = \frac{T}{2}\int_\Omega R_C:  \frac{\partial^2  f(\Qvec)}{\partial \Qvec^2}:\Pvec  \dx -\kappa\int_\Omega   R_C:\Pvec  \dx, $$ with $ \int_\Omega R_C:  \frac{\partial^2  f(\Qvec)}{\partial \Qvec^2}:\Pvec  \dx=\int_\Omega \sum_{i,j,k,l} R_{C,ij} \frac{\partial^2  f (\Qvec)}{\partial \Qvec_{ij} \partial\Qvec_{kl} } \Pvec_{kl}  \dx.   $
 The fact that $f$ is smooth on $\mathcal{Q}_{phy}(\widetilde{\epsilon})$, and the Cauchy-Schwarz inequality lead to 
\begin{align*}
    \int_\Omega \sum_{i,j,k,l} R_{C,ij} \frac{\partial^2  f (\Qvec)}{\partial \Qvec_{ij} \partial\Qvec_{kl} } \Pvec_{kl}  \dx &
    \leq 
    \Big|\Big|\frac{\partial^2  f (\Qvec)}{\partial\Qvec^2}\Big|\Big|_{L^\infty(\Omega)}
    \int_\Omega \sum_{i,j,k,l} \abs{R_{C,ij}} \abs{\Pvec_{kl}}  \dx 
    \\& \leq \Big|\Big|\frac{\partial^2  f (\Qvec)}{\partial\Qvec^2}\Big|\Big|_{L^\infty(\Omega)}
    \norm{R_C}_{L^2(\Omega)} \norm{P}_{L^2(\Omega)}.
\end{align*}
Also, $ \int_\Omega   R_C:\Pvec \, \dx  \leq  \norm{R_C}_{L^2(\Omega)} \norm{P}_{L^2(\Omega)}. $
Therefore, $D\mathcal{B}_{BM}(Q; R_\C,\bullet)   \in \mathcal{L}(L^2(\Omega,  {S}),\mathbb{R}),$ and for $\norm{R_\C}_1=1,$ the following bound is achieved
\begin{align}\label{dis-infsup-fbdd-BM}
    \abs{D\mathcal{B}_{BM}(Q; R_\C,\bullet)}\leq \Big|\Big|\frac{\partial^2  f (\Qvec)}{\partial\Qvec^2}\Big|\Big|_{L^\infty(\Omega)}
    +\kappa .
\end{align}
Under the assumption \eqref{cond-elas-const} on the elasticity parameters, we have the uniform ellipticity of $\mathcal{A}(\bullet, \bullet)$ in $\mathcal{H}^1_0(\Omega).$ This, \eqref{dis-infsup-fbdd-LDG} and \eqref{dis-infsup-fbdd-BM} confirm the existence of a unique solution of \eqref{a linearized problem}. From the elliptic regularity discussed in Subsection \ref{higher regularity}, $\Xi \in \mathcal{H}^2(\Omega)\cap \mathcal{H}^1_0(\Omega)$,, and the following regularity estimate holds 
\begin{align}\label{dis-infsup: regularity}
    \norm{\Xi}_{\mathcal{H}^2(\Omega)} \leq C_{\rm reg} \norm{\mathcal{B}_L(R_\C, \cdot)}_{L^2(\Omega)} \leq  C_{\rm reg} C_Q ,
\end{align}
where the positive constant $C_Q:=\begin{cases}
(a+2b\norm{Q}_{L^\infty(\Omega)}+3c \norm{Q}_{L^\infty(\Omega)}^2) & \text{for LDG bulk potential},\\
    \Big(\frac{T}{2} 
    \Big|\Big|\frac{\partial^2  f (\Qvec)}{\partial\Qvec^2}\Big|\Big|_{L^\infty(\Omega)}
    +\kappa \Big) & \text{for BM bulk potential}.
\end{cases} $
Here, $ C_{\rm reg} >0$ is an elliptic regularity constant, and also depends on the elasticity constants $L_1, L_2, L_3$.
 The inf-sup condition in \eqref{continuous-inf-sup}, and \eqref{a linearized problem} yield that there exists $P \in \mathcal{H}^1_0(\Omega)$ with $\norm{P}_1=1$ such that 
\begin{align}\label{dis-infsup-equn}
    \beta&=\beta\norm{R_\C}_1\leq DN(\Qvec;R_\C,P )=\mathcal{A}(R_\C + \Xi,P) \notag \\& \leq C_E \norm{\nabla (R_\C + \Xi)}_{L^2(\Omega)} \norm{P}_1 \leq C_E(\norm{\nabla (R_\C + \I_\C \Xi)}_{L^2(\Omega)}+\norm{\nabla ( \I_\C \Xi-\Xi)}_{L^2(\Omega)}).
\end{align}
Recall the constant $C_E$ from Lemma \ref{coercivity-elstic-energy}$(i)$.
The coercivity of $\mathcal{A}(\cdot, \cdot)$ in Lemma \ref{coercivity-elstic-energy}$(ii)$ implies that there exists $P_\C \in  {S}_\C^0  $ with $\norm{P_\C}_1 =1$ such that
\begin{align*}
   C_{q_1} \norm{\nabla (R_\C + \I_\C \Xi)}_{L^2(\Omega)} \leq \mathcal{A}(R_\C + \I_\C \Xi,P_\C )= DN(\Qvec;R_\C,P_\C )+\mathcal{A}(\I_\C \Xi-\Xi, P_\C  ).
\end{align*}
For the upper bound for the second terms on the right hand side of the above two estimate, we utilize  Lemmas \ref{coercivity-elstic-energy}$(i)$, \ref{Interpolation estimateConforming}$(i)$, and \eqref{dis-infsup: regularity}, and $\|P_C\|_1=1$ to obtain 
$$\mathcal{A}(\I_\C \Xi-\Xi, P_\C  )+ C_{q_1} \norm{\nabla ( \I_\C \Xi-\Xi)}_{L^2(\Omega)} \leq ( C_E + C_{q_1}) C_I h \norm{\Xi}_{\mathcal{H}^{2}(\Omega)} \leq  ( C_E + C_{q_1} ) C_I C_{\rm reg} C_Q  h.$$ 
A combination of  the above three estimates  results in 
\begin{align*}
     (\beta  C_{q_1} / C_E) - ( C_E + C_{q_1} ) C_I C_{\rm reg} C_Q  h\leq DN(\Qvec;R_\C,P_\C )  .
\end{align*}
This  reveals that the discrete inf-sup condition holds for all $\mathcal{T}\in \mathbb{T}(\delta)$ with $0< \delta \leq \delta_1$ with $\delta_1:=  \beta  C_{q_1} / 2C_E ((C_E + C_{q_1}) C_I C_{\rm reg} C_Q).$ 
This completes the proof.
\end{proof}
\begin{rem}[effect of "physical regime"]
    Note that discrete inf-sup condition in Theorem \ref{thm: dis-inf-sup} associated to the BM potential does not need an extra assumption on the "physicality" of the interpolation operator, whereas Lemma \ref{Interpolation estimateConforming}$(ii)$  is crucial to establish the existence and local uniqueness of the discrete solution in Theorem \ref{thm: existance-uniqueness}. 
\end{rem}

\section{Analysis of the nonlinearity}\label{Analysis of nonlinearity}
In this section, we deal with the non-linearities that appear in both the LDG bulk potential and the singular potential, and discuss some properties of it that are  important to establish the existence, local uniqueness of the discrete solution, and the error estimates. 
\subsection{LDG potential}\label{LDG-potential-nonlinear}

For $ R_\C,P_\C \in  {S}_\C^0,$ and $\widetilde{Q} \in {S}$, the Frech\'et derivative of $\mathcal{B}_{LDG}$ at $\widetilde{Q}$ is given by
\begin{align}\label{defn: LDG lipschitz}
    D\mathcal{B}_{LDG}(\widetilde{Q}; R_\C,P_\C):=&\int_\Omega \big(-a \text{tr}(R_\C P_\C)  -  2b \text{tr}(\widetilde{Q}  R_\C P_\C)  \big) \dx \notag\\&+ c \int_\Omega \big((\text{tr}((\widetilde{Q})^2) \text{tr}(R_\C P_\C) +2\text{tr}(\widetilde{Q}R_\C) \text{tr}(\widetilde{Q}P_\C)) \big) \dx. 
\end{align}
\begin{lem}[Lipschitz continuity of $D\mathcal{B}_{LDG}$]\label{Lipschitz LDG}
For the regular solution $Q \in \mathcal{H}^2(\Omega)$, its interpolation $\I_\C Q $ from Lemma  \ref{Interpolation estimateConforming}$(i),$ and $\widetilde{Q}_\C:=\I_\C Q + Q_\C$ for any $Q_\C\in  {S}_\C^0,$
 the following Lipschitz continuity of $D\mathcal{B}_{LDG}$ holds for $Q_\C^{(1)}, Q_\C^{(2)} \in B_{ {S}_\C^0}(0,1):=\{P_C\in  {S}_\C^0: \, \norm{P_\C}_1 < 1   \}$,
\begin{align*}
    \norm{ D\mathcal{B}_{LDG}(\widetilde{Q}_\C^{(1)})-D\mathcal{B}_{LDG}(\widetilde{Q}_\C^{(2)})}_{\mathcal{L}({( {S}_\C^0)}^*,  {S}_\C^0)}
    \leq L\norm{Q_\C^{(1)} -Q_\C^{(2)}}_{1} 
\end{align*}
with the  Lipschitz constant $ \displaystyle L:=
2b  C_S^3(3)+6c C_S^4(4)(1+( 1+C_I h) \norm{Q}_{\mathcal{H}^2(\Omega)}).   $
\end{lem}
\begin{proof}
For $Q_\C^{(1)},Q_\C^{(2)} \in B_{ {S}_\C^0}(0,1)$ and $ R_\C,P_\C \in  {S}_\C^0,$ the definition of $D\mathcal{B}_{LDG}$  in \eqref{defn: LDG lipschitz}, a cancellation of the linear terms, and a re-arrangement of the remaining terms imply
\begin{align}\label{thm: existence-uniqueness-1}
& D\mathcal{B}_{LDG}(\widetilde{Q}_\C^{(1)}; R_\C,P_\C)-D\mathcal{B}_{LDG}(\widetilde{Q}_\C^{(2)}; R_\C,P_\C) 
\notag \\& = 2b\int_\Omega   \text{tr}((Q_\C^{(2)}-Q_\C^{(1)})R_\C P_\C)  \dx+c \int_\Omega  \text{tr}((Q_\C^{(1)} + \I_\C Q)^2-(Q_\C^{(2)}+\I_\C Q)^2) \, \text{tr}(R_\C P_\C) \dx \notag \\& \quad+2c\int_\Omega \big(\text{tr}((Q_\C^{(1)}-Q_\C^{(2)}) R_\C) \text{tr}((Q_\C^{(1)} +\I_\C Q)P_\C)+\text{tr}((Q_\C^{(2)}+\I_\C Q) R_\C) \text{tr}((Q_\C^{(1)}-Q_\C^{(2)})P_\C) \big) \dx \notag \\&=:T_1+T_2+T_3.
\end{align}
Using the  properties of the trace, the Hölder's  inequality  and the Sobolev embedding $ \mathcal{H}^1(\Omega) \hookrightarrow L^3(\Omega;S)$ for $\Omega \subset \mathbb{R}^3$ results in  
\begin{align*}
  T_1
   \leq  2b \int_\Omega  \abs{Q_\C^{(2)}-Q_\C^{(1)}}_F \abs{R_\C}_F \abs{P_\C}_F \dx  &\leq 2b \norm{Q_\C^{(2)}-Q_\C^{(1)}}_{L^3(\Omega)} \norm{R_\C}_{L^3(\Omega)} \norm{P_\C}_{L^3(\Omega)} 
   \\& \leq 2bC_S^3(3)\norm{Q_\C^{(2)}-Q_\C^{(1)}}_{1} \norm{R_\C}_{1} \norm{P_\C}_{1} ,
\end{align*}
where $C_S(p)>0$ is the constant associated to the Sobolev embedding $ \mathcal{H}^1(\Omega) \hookrightarrow L^p(\Omega;S)$.
To handle the $T_2$ and $T_3$ terms, we utilize the properties of trace, the Hölder's  inequality,  and Sobolev embedding $ \mathcal{H}^1(\Omega) \hookrightarrow L^4(\Omega;S) $ for $\Omega \subset \mathbb{R}^3$, and obtain 
\begin{align*}
    T_2&\leq c \int_\Omega \abs{Q_\C^{(1)}-Q_\C^{(2)}}_F  \abs{Q_\C^{(1)}+Q_\C^{(2)} +2\I_\C Q }_F \abs{R_\C}_F \abs{P_\C}_F \dx  \\& \leq c\norm{Q_\C^{(1)}-Q_\C^{(2)}}_{L^4(\Omega)} \norm{Q_\C^{(1)}+Q_\C^{(2)} +2\I_\C Q}_{L^4(\Omega)} \norm{R_\C}_{L^4(\Omega)} \norm{P_\C}_{L^4(\Omega)}
    \\&  \leq c C_S^4(4)
    \norm{Q_\C^{(1)}-Q_\C^{(2)}}_1 (\norm{Q_\C^{(1)}}_1 +\norm{Q_\C^{(2)}}_1+2\norm{\I_\C Q}_1)\norm{R_\C}_1 \norm{P_\C}_1 ,
\end{align*}
and
\begin{align*}
    T_3&\leq 2c \int_\Omega \abs{Q_\C^{(1)}-Q_\C^{(2)}}_F ( \abs{Q_\C^{(1)}+\I_\C Q }_F+ \abs{Q_\C^{(2)}+\I_\C Q }_F  )\abs{R_\C}_F \abs{P_\C}_F \dx 
    \\&  \leq 2c
 C_S^4(4)   \norm{Q_\C^{(1)}-Q_\C^{(2)}}_1 (\norm{Q_\C^{(1)}}_1 +\norm{Q_\C^{(2)}}_1+2\norm{\I_\C Q}_1)\norm{R_\C}_1 \norm{P_\C}_1. 
\end{align*}
Note that $\norm{Q_\C^{(1)}}_1 \leq 1$ and  $\norm{Q_\C^{(2)}}_1 \leq 1.$ For $\norm{R_\C}_1=1=\norm{P_\C}_1,$ a combination of the estimates for $T_1, T_2, T_3$ in \eqref{thm: existence-uniqueness-1}, and the bound 
$\norm{\I_\C Q}_1\leq \norm{\I_\C Q -Q}_1+\norm{ Q}_1 \leq (C_I h+1)\norm{ Q}_{\mathcal{H}^2(\Omega)}$
imply
\begin{align*}
  &D\mathcal{B}_{LDG}(\widetilde{Q}_\C^{(1)}; R_\C,P_\C)-D\mathcal{B}_{LDG}(\widetilde{Q}_\C^{(2)}; R_\C,P_\C) \\& \leq (2bC_S^3(3)+6c C_S^4(4)(1+(1+C_I h) \norm{ Q}_{\mathcal{H}^2(\Omega)}) ) \norm{Q_\C^{(1)} - Q_\C^{(2)}}_1 .
\end{align*}
This concludes the proof.
\end{proof}
\begin{lem}[Perturbation control of $D\mathcal{B}_{LDG}$]\label{LDG inf-sup-relation} For the regular solution $Q\in \mathcal{H}^2(\Omega)$ and its interpolation $\I_\C Q$ from Lemma  \ref{Interpolation estimateConforming}$(i),$ and $ R_\C,P_\C \in  {S}_\C^0,$ it holds
    \begin{align*}
     &D\mathcal{B}_{LDG}(\I_\C Q; R_\C,P_\C) -D\mathcal{B}_{LDG}(Q; R_\C,P_\C)\notag\\& \leq  (2b C_S^3(3) + 3c (2+C_I h) C_S^4(4) \norm{Q}_{\mathcal{H}^2(\Omega)})\norm{Q-\I_\C Q}_1\norm{R_\C}_1  \norm{P_\C}_1.
\end{align*}
\end{lem}
\begin{proof}
The definition of $D\mathcal{B}_{LDG}$  in \eqref{defn: LDG lipschitz}, and a rearrangement of terms yield
\begin{align*}
    D\mathcal{B}_{LDG}(\I_\C Q; R_\C,P_\C)& -D\mathcal{B}_{LDG}(Q; R_\C,P_\C)=-2b \int_\Omega \text{tr}((\I_\C Q-Q)R_\C P_\C)\dx\\& +c  \int_\Omega \text{tr}(((\I_\C Q)^2-Q^2)R_\C P_\C)\dx +2c  \int_\Omega (\text{tr}((\I_\C Q-Q)R_\C) \text{tr}(\I_\C Q P_\C)\\& + \text{tr}(Q R_\C) \text{tr}((\I_\C Q-Q) P_\C))\dx =:T_1+T_2+T_3.
\end{align*}
The terms $T_1,T_2,T_3$ are controlled utilizing analogous   technique as in Lemma \ref{Lipschitz LDG}. We use  trace properties, the Hölder's  inequality,   Sobolev embedding results $ \mathcal{H}^1(\Omega) \hookrightarrow L^p(\Omega;S),$ $ p=3,4$ for $\Omega \subset \mathbb{R}^3$, and Lemma \ref{Interpolation estimateConforming}$(i)$ and obtain the following boundedness estimates
\begin{align*}
   &T_1 \leq 2b C_S^3(3) \norm{Q-\I_\C Q}_1\norm{R_\C}_1  \norm{P_\C}_1,\\&
   T_2 \leq c (2+C_I h) C_S^4(4) \norm{Q}_{\mathcal{H}^2(\Omega)}\norm{Q-\I_\C Q}_1\norm{R_\C}_1  \norm{P_\C}_1,
   \\&
   T_3 \leq 2c (2+C_I h) C_S^4(4) \norm{Q}_{\mathcal{H}^2(\Omega)}\norm{Q-\I_\C Q}_1\norm{R_\C}_1  \norm{P_\C}_1.
\end{align*}
A combination of the above three estimates completes the proof.
\end{proof}
\begin{lem}[Perturbation control of $\mathcal{B}_{LDG}$]\label{perturbation B-LDG}
For the regular solution $Q\in \mathcal{H}^2(\Omega)$ and its interpolation $\I_\C Q$ from Lemma  \ref{Interpolation estimateConforming}$(i),$ and $ P_\C \in  {S}_\C^0,$ it holds
    \begin{align*}
    \mathcal{B}_{LDG}(\I_\C Q,P_\C) -\mathcal{B}_{LDG}( Q,P_\C)\leq C_B \norm{\I_\C Q - Q}_1\norm{P_\C}_1  ,
\end{align*}
where $C_B:= a+bC_S^3(3)(2+C_I h)\norm{\Qvec}_{\mathcal{H}^2(\Omega)}  +2c C_S^4(4) (3 +2C_I^2h^2)\norm{\Qvec}_{\mathcal{H}^2(\Omega)}^2  .$
\end{lem}
\begin{proof}
Recall the definition of $\mathcal{B}_{LDG}$ from \eqref{nonlinear pde}.
We apply the  Taylor series expansion around the exact solution $Q$ and in the direction of $R:=\I_\C Q -Q \in S,$
   \begin{align*}
      & \mathcal{B}_{LDG}(\I_\C Q,P_\C) -\mathcal{B}_{LDG}( Q,P_\C)\\&= D\mathcal{B}_{LDG}( Q;R, P_\C) +\frac{1}{2}D^2\mathcal{B}_{LDG}( Q;R,R, P_\C) +\frac{1}{6}D^3\mathcal{B}_{LDG}( Q;R,R,R, P_\C) .
   \end{align*}
   Recall the definition of the first order derivative $D\mathcal{B}_{LDG}( Q;\bullet, \bullet)$ from \eqref{defn: LDG lipschitz}.
   A direct calculation of the higher order derivatives of $ \mathcal{B}_{LDG}$ in the above expansion yields 
   \begin{align*}
      &
D^2\mathcal{B}_{LDG}( Q;R,R, P_\C):=\int_\Omega \big(-2b \text{tr}(R^2 P_\C)+4c \text{tr}(QR )\text{tr}(R P_\C)+2c \text{tr}(R^2 )\text{tr}(Q P_\C)\big) \dx,\\&
D^3\mathcal{B}_{LDG}( Q;R,R,R, P_\C):=\int_\Omega 6c \text{tr}(R^2 )\text{tr}(R P_\C) \dx.
   \end{align*}
The estimates for $D\mathcal{B}_{LDG}( Q;\bullet, \bullet)$, and the above two higher order derivatives are derived from trace properties and the Hölder's  inequality combined with the Sobolev embeddings $ \mathcal{H}^1(\Omega) \hookrightarrow L^p(\Omega;S),$ $ p=2,3,4$ for $\Omega \subset \mathbb{R}^3$, and it is obtained that
 \begin{align*}
    & D\mathcal{B}_{LDG}( Q;R, P_\C) \leq (a   +2bC_S^3(3)\norm{Q}_1 +3cC_S^4(4) \norm{Q}_1^2 ) \norm{R}_1 \norm{P_\C}_1,\\&
    D^2\mathcal{B}_{LDG}( Q;R,R, P_\C)  \leq 2(bC_S^3(3)+3c C_S^4(4)\norm{Q}_1)\norm{R}_1^2\norm{P_\C}_1,\\&
  D^3\mathcal{B}_{LDG}( Q;R,R,R, P_\C) \leq   6c C_S^4(4)\norm{R}_1^3\norm{P_\C}_1.
   \end{align*}
Therefore, it holds
\begin{align*}
  &  \mathcal{B}_{LDG}(\I_\C Q,P_\C) -\mathcal{B}_{LDG}( Q,P_\C) \\& \leq \big(a+b C_S^3(3) (2\norm{Q}_1 +\norm{R}_1)+2c C_S^4(4)(3\norm{Q}_1^2 +2\norm{R}_1^2) \big) \norm{R}_1 \norm{P_\C}_1.
\end{align*}
Here we use Young's inequality for a simplification of the last term on the right hand side of the inequality. This and Lemma \ref{Interpolation estimateConforming}$(i)$ concludes the proof.
\end{proof}
\subsection{BM potential}
For $ R_\C,P_\C \in  {S}_\C^0,$ and $\widetilde{Q} \in \mathcal{Q}_{phy}(\widetilde{\epsilon})$, the Frech\'et derivative of $\mathcal{B}_{BM}$ at $\widetilde{Q}$ is given by
\begin{align}\label{defn: BM lipschitz}
    D\mathcal{B}_{BM}(\widetilde{Q}; R_\C,P_\C):=  \frac{T}{2}\int_\Omega R_\C:  \frac{\partial^2  f(\widetilde{Q})}{\partial \Qvec^2}:\Pvec_\C  \dx -\kappa \int_\Omega   R_\C :\Pvec_\C  \dx. 
\end{align}
\begin{lem}[Lipschitz continuity of $D\mathcal{B}_{BM}$]\label{Lipschitz BM}
For the regular solution $Q\in \mathcal{H}^2(\Omega)\cap \mathcal{Q}_{phy}(\widetilde{\epsilon})$ and its interpolation $\I_\C Q \in \mathcal{Q}_{phy}(\widetilde{\epsilon}_1)  $ from Lemma  \ref{Interpolation estimateConforming}$(ii),$ and $\widetilde{Q}_\C:=\I_\C Q + Q_\C$ for any $Q_\C\in  {S}_\C^0,$
 the following Lipschitz continuity of $D\mathcal{B}_{BM}$ holds for  $Q_\C^{(1)}, Q_\C^{(2)} \in \mathcal{D}:=\{R_\C \in B_{ {S}_\C^0} (0, 1)| \, R_\C +\I_\C \Qvec \in \mathcal{Q}_{phy}(\widetilde{\epsilon_1}) \}\subset {S}_\C^0$, 
    \begin{align*}
    \norm{ D\mathcal{B}_{BM}(\widetilde{Q}_\C^{(1)})-D\mathcal{B}_{BM}(\widetilde{Q}_\C^{(2)})}_{\mathcal{L}({( {S}_\C^0)}^*,  {S}_\C^0)}
    \leq L\norm{Q_\C^{(1)} -Q_\C^{(2)}}_{1}
\end{align*}
for all $\mathcal{T} \in \mathbb{T}(\delta_0).$ Here $L:=\frac{T   C_S^3(3) }{2}  \sup_{P \in \mathcal{Q}_{phy}(\widetilde{\epsilon_1}) }\Big| \frac{\partial^3  f (P)}{\partial \Qvec^3 }\Big| $.
 Recall $\delta_0$ from Lemma \ref{Interpolation estimateConforming}$(ii)$.
\end{lem}
\begin{proof}
Let $Q_\C^{(1)}, Q_\C^{(2)} \in \mathcal{D}$ and $R_\C,P_\C \in  {S}_\C^0$ with $\norm{R_\C}_1=1=\norm{P_\C}_1.$    The definition of $D\mathcal{B}_{BM}$ in \eqref{defn: BM lipschitz}, and a cancellation of the linear terms imply
\begin{align}\label{equn1: BM int1}
 & D\mathcal{B}_{BM}(\widetilde{Q}_\C^{(1)}; R_\C,P_\C)- D\mathcal{B}_{BM}(\widetilde{Q}_\C^{(2)}; R_\C,P_\C) \notag\\&=\frac{T  }{2} \int_\Omega R_\C : \Big( \frac{\partial^2  f(Q_\C^{(1)}+\I_\C Q)}{\partial \Qvec^2}- \frac{\partial^2  f(Q_\C^{(2)}+\I_\C Q)}{\partial \Qvec^2} \Big):\Pvec_\C  \dx
 \notag\\&
 =\frac{T  }{2}  \int_\Omega\sum_{i,j,k,l} R_{C,ij}  \Big( \frac{\partial^2  f (Q_\C^{(1)}+\I_\C Q)}{\partial \Qvec_{ij} \partial\Qvec_{kl} } -\frac{\partial^2  f (Q_\C^{(2)}+\I_\C Q)}{\partial \Qvec_{ij} \partial\Qvec_{kl} }  \Big)
  P_{C,kl}   \dx.
\end{align}
The convexity of $\mathcal{Q}_{phy}$ implies that for $\widetilde{Q}_\C^{(1)}, \widetilde{Q}_\C^{(2)} \in \mathcal{Q}_{phy}(\widetilde{\epsilon_1}),$ $ (t {Q}_\C^{(1)}+(1-t){Q}_\C^{(2)} +\I_\C Q) \in  \mathcal{Q}_{phy}(\widetilde{\epsilon_1})$ for $t\in [0,1]$.
Now the integral formula for the remainder term of the Taylor expansion,  the Hölder inequality, and the Sobolev embedding $\mathcal{H}^1(\Omega) \hookrightarrow L^3(\Omega;S)$ result in
\begin{align*}
    & \int_\Omega\sum_{i,j,k,l} R_{C,ij}  \Big( \frac{\partial^2  f (Q_\C^{(1)}+\I_\C Q)}{\partial \Qvec_{ij} \partial\Qvec_{kl} } -\frac{\partial^2  f (Q_\C^{(2)}+\I_\C Q)}{\partial \Qvec_{ij} \partial\Qvec_{kl} }  \Big)
  P_{C,kl}   \dx \notag\\& =\int_\Omega\sum_{i,j,k,l,m,n} R_{C,ij}  \int_0^1 \frac{\partial^3  f (t Q_\C^{(1)}+(1-t)Q_\C^{(2)} +\I_\C Q)}{\partial \Qvec_{ij} \partial\Qvec_{kl}\partial\Qvec_{mn} }( Q_{\C,mn}^{(1)} -Q_{\C,mn}^{(2)}) d {\rm t} \,
  P_{C,kl}   \dx \notag\\& =C_1\int_\Omega \sum_{i,j,k,l,m,n} \abs{R_{C,ij}} \abs{Q_{\C,mn}^{(1)} -Q_{\C,mn}^{(2)}} \abs{ P_{C,kl}}  \dx
\notag\\& \leq C_1 \norm{R_\C}_{L^3(\Omega)}\norm{Q_\C^{(1)} - Q_\C^{(2)}}_{L^3(\Omega)}\norm{P_\C}_{L^3(\Omega)}
 \leq C_1 C_S^3(3)\norm{Q_\C^{(1)} - Q_\C^{(2)}}_1 \norm{R_\C}_1  \norm{P_\C}_1,
\end{align*}
where $ C_1:= \sup_{P \in \mathcal{Q}_{phy}(\widetilde{\epsilon_1}) } \Big| \frac{\partial^3  f (P)}{\partial \Qvec^3 }\Big| $. 
This together with the unit norms $\norm{R_\C}_1=1=\norm{P_\C}_1$ applied to the above estimate  concludes the proof.
\end{proof}
\begin{lem}[Perturbation control of $D\mathcal{B}_{BM}$]\label{BM-inf-sup-relation} For the regular solution $Q\in \mathcal{H}^2(\Omega)\cap \mathcal{Q}_{phy}(\widetilde{\epsilon})$ and its interpolation $\I_\C Q \in \mathcal{Q}_{phy}(\widetilde{\epsilon_1})$ from Lemma  \ref{Interpolation estimateConforming}$(ii),$   and $ R_\C,P_\C \in  {S}_\C^0,$ it holds
    \begin{align*}
     &D\mathcal{B}_{BM}(\I_\C Q; R_\C,P_\C) -D\mathcal{B}_{BM}(Q; R_\C,P_\C) \leq  C_1 C_S^3(3) \norm{Q-\I_\C Q}_1\norm{R_\C}_1  \norm{P_\C}_1.
\end{align*}
Recall $ C_1:= \sup_{P \in  \mathcal{Q}_{phy}(\widetilde{\epsilon_1}) } \Big| \frac{\partial^3  f (P)}{\partial \Qvec^3 }\Big|$ from the proof of Lemma \ref{Lipschitz BM}.
\end{lem}
\begin{proof}
For the BM potential, $\mathcal{B}=\mathcal{B}_{BM}$,
\begin{align*}
    D\mathcal{B}_{BM}(\I_\C Q; R_\C,P_\C) -D\mathcal{B}_{BM}(Q; R_\C,P_\C)= \frac{T   }{2}  \int_\Omega R_\C:  \Big(\frac{\partial^2  f(\I_\C Q)}{\partial \Qvec^2}-\frac{\partial^2  f(Q)}{\partial \Qvec^2} \Big):\Pvec_\C   \dx .
\end{align*}
Now a similar argumentation as in \eqref{equn1: BM int1}, using the integral formula for the remainder term of the Taylor expansion, the Hölder's  inequality, and the Sobolev embedding $\mathcal{H}^1(\Omega) \hookrightarrow L^3(\Omega;S)$ imply
\begin{align*}
    \int_\Omega R_\C:  \Big(\frac{\partial^2  f(\I_\C Q)}{\partial \Qvec^2}-\frac{\partial^2  f(Q)}{\partial \Qvec^2} \Big):\Pvec_\C  \dx &=\int_\Omega\sum_{i,j,k,l} R_{C,ij} \Big( \frac{\partial^2  f (\I_\C Q)}{\partial \Qvec_{ij} \partial\Qvec_{kl} } -\frac{\partial^2 f (Q)}{\partial \Qvec_{ij} \partial\Qvec_{kl} } \Big)  P_{C,kl}  \dx \notag\\& 
    \leq  C_1 C_S^3(3)   \norm{Q-\I_\C Q  }_{1} \norm{R_\C}_{1} \norm{P_\C}_{1}.
\end{align*}
This completes the proof.
\end{proof}
\begin{lem}[Perturbation control of $\mathcal{B}_{BM}$]\label{perturbation B-BM}
For the regular solution $Q\in \mathcal{H}^2(\Omega) \cap \mathcal{Q}_{phy}(\widetilde{\epsilon})$ and its interpolation $\I_\C Q  \in  \mathcal{Q}_{phy}(\widetilde{\epsilon_1})$ from Lemma  \ref{Interpolation estimateConforming}$(ii),$ and $ P_\C \in  {S}_\C^0,$ it holds
    \begin{align*}
    \mathcal{B}_{BM}(\I_\C Q,P_\C) -\mathcal{B}_{BM}( Q,P_\C)\leq C_B \norm{\I_\C Q - Q}_1 \norm{P_\C}_1 ,
\end{align*}
where $C_B:=\kappa + \frac{T}{2} \sup_{P \in \mathcal{Q}_{phy}(\widetilde{\epsilon_1})} \Big| \frac{\partial^2  f (P)}{\partial \Qvec^2 } \Big|  .$
\end{lem}
\begin{proof}
 The definition of $\mathcal{B}_{BM}$ from \eqref{nonlinear pde} and a rearrangement of terms lead to
\begin{align}\label{equn1: BM int}
      & \mathcal{B}_{BM}(\I_\C Q,P_\C) -\mathcal{B}_{BM}( Q,P_\C) \notag\\&= 
       \frac{T }{2}  \int_\Omega \Big(\frac{\partial f (\I_\C Q)}{\partial \Qvec }-\frac{\partial f ( Q)}{\partial \Qvec } \Big): P_\C \, d\mathbf{x}  - \kappa \int_\Omega  (\I_\C Q-Q): P_\C \, d\mathbf{x}=:T_1+T_2 .
   \end{align}
 The Cauchy-Schwarz inequality, and the Sobolev embedding imply
\begin{align*}
 T_1 \leq  \kappa    \norm{Q-\I_\C Q}_{L^2(\Omega)}\norm{P_\C}_{L^2(\Omega)}\leq  \kappa    \norm{Q-\I_\C Q}_{1}\norm{P_\C}_{1}.
\end{align*}
Define $C_2:=\displaystyle \sup_{P \in \mathcal{Q}_{phy}(\widetilde{\epsilon_1})} \Big| \frac{\partial^2  f (P)}{\partial \Qvec^2 } \Big|$. The integral formula for the remainder term of the Taylor expansion and  the Cauchy-Schwarz inequality result in
\begin{align*}
 T_2  &=\frac{T }{2}  \int_\Omega
    \sum_{i,j,k,l=1}^3 \int_0^1 \frac{\partial^2 f (t \I_\C Q +(1-t)Q)}{\partial \Qvec_{ij} \partial \Qvec_{kl} } (\I_\C Q_{ij}- Q_{ij}) \, d {\rm t} \, P_{\C,kl} \dx
 \\&   \leq \frac{T C_2 }{2}   \norm{\I_\C Q- Q}_{L^2(\Omega)} \norm{P_\C}_{L^2(\Omega)} \leq \frac{T C_2 }{2} \norm{\I_\C Q- Q}_{1} \norm{P_\C}_{1}.
\end{align*}
A combination of the estimates for $T_1, T_2$ in \eqref{equn1: BM int} completes the proof.
\end{proof}
\section{Existence and local uniqueness of discrete solution}\label{existence-uniqueness}
In this section, we present our main result on the existence, local uniqueness of the discrete solution, and the {\it a priori} error estimates.
\begin{thm}\label{thm: existance-uniqueness}
   Given a regular solution $Q\in \mathcal{H}^2(\Omega) $ of \eqref{nonlinear pde}  with $Q_b \in H^{3/2}(\partial \Omega;S)$, and suitably restricted values of  the elastic constants $L_1, L_2, L_3$ as $  0<L_1, -L_1< L_3< 2L_1, \text{ and } -\frac{3}{5} L_1 -\frac{1}{10} L_3 <L_2,$  there exists a $\delta >0$ and a constant $C_{q}$ such that  for all $\mathcal{T}\in \mathbb{T}(\delta),$ there exists a unique local discrete solution $Q_\C \in  {S}_\C$ to \eqref{discrete problem} near $Q$ and satisfies the error estimates
    \begin{align*}
        \norm{Q - Q_\C}_1 \leq C_q \norm{Q - \I_\C Q}_1 \approx O(h),
    \end{align*}
    where 
    $$C_q:=\begin{cases}
1+ 2(C_E+ a+bC_S^3(3)(2+C_I )\norm{\Qvec}_{\mathcal{H}^2(\Omega)}  +2c C_S^4(4) (3 +2C_I^2)\norm{\Qvec}_{\mathcal{H}^2(\Omega)}^2)/\beta_1&\text{ for LDG potential},\\
     1+  2 \Big(C_E+\kappa + \frac{T}{2} \sup_{P \in \mathcal{Q}_{phy}(\widetilde{\epsilon_1})} \Big| \frac{\partial^2  f (P)}{\partial \Qvec^2 } \Big| \Big) /\beta_1 & \text{ for BM potential}.
    \end{cases}$$
Recall  $\beta_1 $ from Theorem \ref{thm: dis-inf-sup}, $C_E  $ from Lemma \ref{coercivity-elstic-energy}, $C_I$ from Lemma \ref{Interpolation estimateConforming}, $C_S$ is a constant from Sobolev embedding results. 
\end{thm}
The proof uses the Kantorovich theorem. We now state and verify the assumptions of this theorem in our discrete setup.
  \begin{thm}[Kantorovich]\cite{Kantorovich_1948,zeidler}\label{Kantorovich theorem}
		Let $Z,Y$ be Banach spaces, and  $L(Y, Z)$ denote the Banach space of bounded linear operators of $Y$ into $Z$. 	Suppose that the mapping $\mathscr{N} : {\rm D} \subset Z \rightarrow Y$ is Fr\'echet differentiable on an open convex set ${\rm D}$, and the derivative  
			$D\mathscr{N}(\cdot)$ is Lipschitz continuous on ${\rm D}$ with Lipschitz constant $L$. 
			For a fixed starting point $x^0 \in {\rm D}, $ the inverse $D\mathscr{N}(x^0)^{-1}$ exists as a continuous operator on $Z.$ The real numbers $a_1$ and $b_1$ are chosen such that 
			\begin{align}\label{eq:4.9}
			\norm{D\mathscr{N}(x^0)^{-1}}_{L(Y;Z )} \leq a_1 \quad \text{and} \quad \norm{D\mathscr{N}(x^0)^{-1}\mathscr{N}(x^0)}_{Z} \leq b_1
			\end{align}
			and $h^*:=a_1 b_1 L\leq \frac{1}{2}.$ Suppose, the first approximation $x^1:= {x^0}- D\mathscr{N}(x^0)^{-1}\mathscr{N}(x^0) $ has the property that the closed ball $\overline{B}_Z(x^1,r):= \{x\in Z|\, \norm{x-x^1}_Z \leq r  \}$ lies within the domain of definition ${\rm D},$  where $r= \frac{1-\sqrt{1-2h^*}}{a_1 L}-b_1.$
			Then the following are true.
			\begin{enumerate}
				\item {Existence and uniqueness}. There exists a solution, $x\in \overline{B}_Z(x^1,r)$ to $\mathscr{N}(x)=0,$ and the solution is unique on $\overline{B}_Z(x^0,r^*)\cap {\rm D},$ that is, on a suitable neighborhood of the initial point, ${x^0}$, with $r^*=\frac{1+\sqrt{1-2h^*}}{a_1 L}. $ 
				\item Convergence of Newton's method. The Newton's scheme with initial iterate, $x^0,$ leads to a sequence,  $x^n:= {x^{n-1}}- D\mathscr{N}(x^{n-1})^{-1}\mathscr{N}(x^{n-1}) $, in $\overline{B}_Z(x^0,r^*)$, which converges to, $x,$ with error bound $	\norm{x^n-x}_{Z}\leq \frac{(1-(1-2h^*)^{\frac{1}{2}})^{2^n}}{{2^n}a_1 L}, \quad n=0,1 \dots.$
			\end{enumerate}
		\end{thm}
  \medskip

For a suitable interpolation $\I_\C \Qvec $ of the regular solution  $Q\in \mathcal{H}^2(\Omega)$ of \eqref{nonlinear pde}, define the map 
$ \mathcal{N}:\mathcal{D}\subset  {S}_\C^0 \rightarrow ( {S}_\C^0)^*$  by 
\begin{align}\label{defn: N-tilde}
    \mathcal{N}(R_\C):=  N(R_\C+\I_\C Q) \text{  for all  }R_\C \in \mathcal{D}.
\end{align}
For the LDG bulk potential,  the interpolation operator is considered as in Lemma \ref{Interpolation estimateConforming}$(i),$ and the domain of definition of the map $\mathscr{N}(\bullet)$  is $ \mathcal{D}: =B_{ {S}_\C^0}(0,1):=\{P_C\in  {S}_\C^0: \, \norm{P_\C}_1 < 1   \}$. Whereas, in the case of singular potential, for a regular solution $\Qvec \in
\mathcal{Q}_{phy}(\widetilde{\epsilon}), \widetilde{\epsilon}>0$ of \eqref{nonlinear pde},  we consider the physical interpolation $\I_\C \Qvec  \in \mathcal{Q}_{phy}(\widetilde{\epsilon_1}), \widetilde{\epsilon_1}>0$ (see  Lemma \ref{Interpolation estimateConforming}$(ii)$). 
 Now   define  the translated set $\Tilde{\mathcal{D}}:=\mathcal{Q}_{phy}(\widetilde{\epsilon_1})- \I_\C \Qvec,$ 
 and choose   $\mathcal{D}:=\{R_\C \in B_{ {S}_\C^0} (0, 1)\subset {S}_\C^0 | \, R_\C +\I_\C \Qvec \in \mathcal{Q}_{phy}(\widetilde{\epsilon_1}) \}$. Since $\Tilde{\mathcal{D}}$ is a closed, bounded, and convex subset subset of $S,$ $\mathcal{D}$ is open, bounded and convex set in ${S}_\C^0.$ 

   We work on the mesh parameter regime $h< \delta_2:=\min(\delta_0,\delta_1).$ Recall $\delta_0$ from Lemma \ref{Interpolation estimateConforming}$(ii)$ and $\delta_1$ from Theorem \ref{thm: dis-inf-sup}. 
 With out loss of generality, for the case of LDG potential, we assume $\delta_0=1.$   We utilize this ($h<\delta_0=1$) in the bounds associated to the nonlinearity in Section \ref{LDG-potential-nonlinear}.

The Kantorovich theorem  is applied for the discrete map $\mathcal{N}.$ Now we proceed to verify the crucial assumptions  for $\mathcal{N}$ in Theorem \ref{Kantorovich theorem}.

\medskip

\noindent Step 1. {\bf (Lipschitz continuity of $D\mathcal{N}$)} For $Q_\C, R_\C,P_\C \in  {S}_\C^0,$  the Fr\'echet derivative of $\mathcal{N}$ at $Q_\C$ is defined as, 
\begin{align}\label{equn: Fderivative 1}
  &  D\mathcal{N}(Q_\C; R_\C,P_\C):=  D {N}(Q_\C+\I_\C Q; R_\C,P_\C)= \mathcal{A}(R_\C,P_\C) +D\mathcal{B}(Q_\C+\I_\C Q; R_\C,P_\C).
\end{align}
Recall the  Frech\'et derivatives for the  
 LDG bulk potential, $\mathcal{B}=\mathcal{B}_{LDG}$, from \eqref{defn: LDG lipschitz}, and for the  
 BM bulk potential, $\mathcal{B}=\mathcal{B}_{BM}$, from \eqref{defn: BM lipschitz}.
Note that for $Q_\C^{(1)},Q_\C^{(2)} \in \mathcal{D},$ a cancellation of the linear term $\mathcal{A}(\bullet,\bullet)$ leads to
\begin{align*}
    D\mathcal{N}(Q_\C^{(1)}; R_\C,P_\C) -D\mathcal{N}(Q_\C^{(2)}; R_\C,P_\C)=D\mathcal{B}(Q_\C^{(1)}; R_\C,P_\C) -D\mathcal{B}(Q_\C^{(2)}; R_\C,P_\C),
\end{align*}
and thus, the Lipschitz continuity of $ D\mathcal{N}$ is determined entirely by the Lipschitz continuity of the nonlinear operator 
$D\mathcal{B}.$ 
Therefore, Lemmas \ref{Lipschitz LDG} and \ref{Lipschitz BM} results in 
\begin{align*}
    \norm{ D\mathcal{N}(Q_\C^{(1)})-D\mathcal{N}(Q_\C^{(2)})}_{\mathcal{L}({( {S}_\C^0)}^*,  {S}_\C^0)}
    \leq L\norm{Q_\C^{(1)} -Q_\C^{(2)}}_{1} ,
\end{align*}
where the  Lipschitz constant $ \displaystyle L:=\begin{cases}
2b  C_S^3(3)+6c C_S^4(4)(1+(1+C_I ) \norm{Q}_{\mathcal{H}^2(\Omega)})   &\text{ for LDG potential}, \\
     \frac{T   C_S^3(3) }{2}  \sup_{P \in \mathcal{Q}_{phy}(\widetilde{\epsilon_1}) }\Big| \frac{\partial^3  f (P)}{\partial \Qvec^3 }\Big|  & \text{ for BM potential}.\\ 
\end{cases}$

\medskip 

\noindent Step 2. {\bf(Inf-sup constant}) For $R_\C, P_\C \in {S}_\C^0$ with $\norm{R_\C}_1=1=\norm{P_\C}_1,$ and a suitable choice of the interpolation $\I_\C Q$ as specified in \eqref{defn: N-tilde}, the definition of the Fr\'echet derivative  $D\mathcal{N}(0)$ from \eqref{equn: Fderivative 1} leads to
\begin{align}\label{inf-sup-relation}
    D\mathcal{N} (0; R_\C,P_\C)&=  D {N}( \I_\C Q; R_\C,P_\C) \notag\\&= D {N}( Q; R_\C,P_\C) +(D\mathcal{B}(\I_\C Q; R_\C,P_\C) -D\mathcal{B}(Q; R_\C,P_\C)).
\end{align}
The perturbation error control in the second term on the right hand side of \eqref{inf-sup-relation} is discussed in Lemmas \ref{LDG inf-sup-relation} and \ref{BM-inf-sup-relation} in  Section \ref{Analysis of nonlinearity}.  These estimates, Theorem \ref{thm: dis-inf-sup}, and Lemma \ref{Interpolation estimateConforming}$(i)$ infer
\begin{align*}
\beta_1 -  C_3 h \leq    D\mathcal{N} (0; R_\C,P_\C)
\end{align*}
with $C_3:=\begin{cases}
 C_I(2b C_S^3(3) + 3c (2+C_I ) C_S^4(4) \norm{Q}_{\mathcal{H}^2(\Omega)})   \norm{Q}_{\mathcal{H}^2(\Omega)}  & \text{ for LDG potential}, \\ \frac{T   C_S^3(3) C_I}{2}  \Big( \sup_{P \in  \mathcal{Q}_{phy}(\widetilde{\epsilon_1})} \Big| \frac{\partial^3  f (P)}{\partial \Qvec^3 }\Big| \Big)   \norm{ Q  }_{\mathcal{H}^2(\Omega)} & \text{ for BM potential.}
\end{cases}$\\
This concludes the invertibility of the the Fr\'echet derivative  $D\mathcal{N}(0)$   for all $\mathcal{T} \in \mathbb{T}(\delta)$ with $0< \delta \leq \delta_3:=\min{\Big(\delta_2, \beta_1/2C_3 \Big)}, $ and 
\begin{align}\label{equn: inf-sup-interpolation}
    \norm{D\mathcal{N}(0)^{-1}}_{\mathcal{L}({( {S}_\C^0)}^*,  {S}_\C^0)} \leq 2/\beta_1.
\end{align}
Therefore, the first estimate in \eqref{eq:4.9} holds for $a_1:=2/\beta_1.$

\medskip
\noindent Step 3. {\bf (residual control)} Now we establish that the residual $ \displaystyle  \norm{\mathcal{N} (0)}_{( {S}_\C^0)^*} =\sup_{P_\C \in B_{ {S}_\C^0}(0,1)} \mathcal{N}(0;P_\C)$ in  \eqref{eq:4.9} is small.
Recall the definition of $\mathcal{N} (\cdot)$ from \eqref{defn: N-tilde}, and use \eqref{nonlinear pde} to obtain 
\begin{align}\label{equn: residual1}
     \mathcal{N} (0; P_\C) &= \mathcal{A}(\I_\C Q,P_\C)+\mathcal{B}(\I_\C Q,P_\C) \notag \\&
     =\mathcal{A}(\I_\C Q - Q, P_\C) + (\mathcal{B}(\I_\C Q,P_\C) -\mathcal{B}( Q,P_\C)).
\end{align}
The perturbation error control for the nonlinear bulk term $\mathcal{B}(\bullet)$ is discussed in Lemma \ref{perturbation B-LDG} (resp. Lemma \ref{perturbation B-BM}) for LDG potential (resp. BM potential) in Section \ref{Analysis of nonlinearity}.
Under the assumption \eqref{cond-elas-const} on the elastic parameters, Lemma \ref{coercivity-elstic-energy}$(i)$, the Cauchy-Schwarz inequality with the  Sobolev bound in $H^1$-norm yield
\begin{align} \label{equn: residual3}
   & \mathcal{A}(\I_\C Q -Q,P_\C) \leq C_E \norm{\nabla (\I_\C Q -Q)}_{L^2(\Omega)} \norm{\nabla P_\C}_{L^2(\Omega)} \leq C_E\norm{Q-\I_\C Q}_{1}\norm{P_\C}_{1}.
   \end{align}
The above estimate, Lemma \ref{perturbation B-LDG}, and $\norm{P_\C}_1 \leq 1$ applied to \eqref{equn: residual1} imply
\begin{align*}
   \mathcal{N} (0; P_\C)= \mathcal{A}(\I_\C Q,P_\C)+\mathcal{B}_{LDG}(\I_\C Q,P_\C)   \leq (C_E +C_B) \norm{\I_\C Q - Q}_1 
\end{align*}
Recall $C_B:= a+bC_S^3(3)(2+C_I )\norm{\Qvec}_{\mathcal{H}^2(\Omega)}  +2c C_S^4(4) (3 +2C_I^2)\norm{\Qvec}_{\mathcal{H}^2(\Omega)}^2$   from Lemma \ref{perturbation B-LDG}.
Similarly, the treatment of the residual associated to the BM potential is obtained from Lemma \ref{perturbation B-BM} with $\norm{P_\C}_1 \leq 1$ and \eqref{equn: residual3}, as follows:
\begin{align*}
 \mathcal{N} (0; P_\C)= \mathcal{A}(\I_\C Q,P_\C)+\mathcal{B}_{BM}(\I_\C Q,P_\C) 
     \leq (C_E+C_B )\norm{\I_\C Q- Q}_{1} 
\end{align*}
with $C_B:=\kappa + \frac{T}{2} \sup_{P \in \mathcal{Q}_{phy}(\widetilde{\epsilon_1})} \Big| \frac{\partial^2  f (P)}{\partial \Qvec^2 } \Big| .$
A combination of the above two estimates  provide the residual bound for \eqref{eq:4.9} as
\begin{align*}
    \norm{\mathcal{N} (0)}_{( {S}_\C^0)^*}\leq a_2 \text{ with }a_2:=(C_E+ C_B)\norm{\I_\C Q- Q}_{1}.
\end{align*}
This and \eqref{equn: inf-sup-interpolation}  leads to $ \norm{D \mathcal{N}(0)^{-1} \mathcal{N}(0)}_{ {S}_\C^0} \leq b_1 $ for $b_1=2a_2/\beta_1.$

\medskip

\noindent Step 4. {\bf ($a_1b_1L\leq 1/2$ and ball condition)} Note that the constant $b_1$ in Step 4 involve interpolation estimate for the exact solution, and hence can be the controlled by tuning the discretization parameter. Therefore, Lemma \ref{Interpolation estimateConforming} implies that  $b_1\leq 1/2$ and $ h^*:= a_1 b_1 L\leq 1/2$ for  all $\mathcal{T} \in \mathbb{T}(\delta)$ such that $0<\delta \leq \delta_4 \leq \min{(\delta_3, \min({1, \beta_1/2L})C_4)}$ with $C_4:=\beta_1/( 4(C_E+C_B) C_I \norm{\Qvec}_{\mathcal{H}^2(\Omega)}).$ 

Next we establish that the closed ball $\bar{B}_{ {S}_\C^0}(Q_\C^1,r)$ with $r:=\frac{1-\sqrt{1-2h^*}}{a_1 L}-b_1$ and centered around the first Newton iterate $Q_\C^1:= 0- D\mathcal{N}(0)^{-1}\mathcal{N}(0) $ lies within the domain of definition, $\mathcal{D}$. Let $P_\C \in \bar{B}_{ {S}_\C^0}(Q_\C^1,r).$ Then this {\it ball condition}, that is, $\bar{B}_{ {S}_\C^0}(Q_\C^1,r) \subset \mathcal{D},$ follows readily from a triangle inequality   as given below
\begin{align*}
    \norm{P_\C}_1 \leq \norm{P_\C -Q_\C^1 }_1 + \norm{Q_\C^1 }_1 \leq r+b_1< 2b_1<1.
\end{align*}

\medskip

\noindent Step 5. {(\bf Conclusions)} The Newton Kantorovich theorem applied to the function $\mathcal{N}: \mathcal{D} \rightarrow ( {S}_\C^0)^*$ leads to the existence of a discrete solution $\Tilde{Q}_\C \in B_{ {S}_\C^0}(0,r),$ and the solution is unique in $\Bar{B}(0, r^*) \cap \mathcal{D}$ with $r^*:=\frac{1+\sqrt{1-2h^*}}{a_1 L}.$ Note that 
\begin{align*}
    \norm{\Tilde{Q}_\C}_1 \leq r <b_1=2 (C_E+ C_B)\norm{\I_\C Q- Q}_{1}/\beta_1.
\end{align*}
For $Q_\C:=\Tilde{Q}_\C +\I_\C Q,$ a triangle inequality and the above estimate results in the following quasi-best approximation result,
\begin{align*}
    \norm{Q - Q_\C}_1 &\leq  \norm{Q - \I_\C Q}_1 + \norm{\Tilde{Q}_\C}_1 \leq C_q \norm{ Q-\I_\C Q}_{1},
\end{align*}
where 
$$C_q:=\begin{cases}
1+ 2(C_E+ a+bC_S^3(3)(2+C_I )\norm{\Qvec}_{\mathcal{H}^2(\Omega)}  +2c C_S^4(4) (3 +2C_I^2)\norm{\Qvec}_{\mathcal{H}^2(\Omega)}^2)/\beta_1&\text{ for LDG potential},\\
     1+  2 \Big(C_E+\kappa + \frac{T}{2} \sup_{P \in \mathcal{Q}_{phy}(\widetilde{\epsilon_1})} \Big| \frac{\partial^2  f (P)}{\partial \Qvec^2 } \Big| \Big) /\beta_1 & \text{ for BM potential}.
    \end{cases}$$ 
{ The boundary condition  $Q_\C|_{\partial \Omega}=\I_\C Q_b$ follows from the fact that $\Tilde{Q}_\C \in {S}_\C^0, $ and $\I_\C Q|_{\partial \Omega}=\I_\C Q_b.$}
    This and the interpolation estimate in Lemma \ref{Interpolation estimateConforming}$(i)$ imply the error estimate $$ \norm{Q - Q_\C}_1 \leq h C_I C_q  \norm{\Qvec}_{\mathcal{H}^2(\Omega)}.$$ 
   { This provides the existence of a discrete solution with error of optimal order $h$. 
   
   Next we briefly examine the local uniqueness. Assume the existence of another discrete solution $Q_\C^*$ such that $ \norm{Q - Q_\C^*}_1 \leq h C_5  \norm{\Qvec}_{\mathcal{H}^2(\Omega)}$ for some constant $C_5 >0.$ Note that $Q_\C^*$ must be physical for the BM potential.
   Therefore, \eqref{defn: N-tilde} implies $\mathcal{N}(Q_\C^*- \I_\C Q)=0$ with  $\norm{Q_\C^*- \I_\C Q}_1 > r^*$ as $\Tilde{Q}_\C $ is the unique solution  of  $\mathcal{N}(\Tilde{Q}_\C)=0$ in $\Bar{B}(0, r^*) \cap \mathcal{D}$.  This leads to a contradiction for sufficiently small mesh-size via the inequality $1 \leq (a_1 L) r^*  < h a_1 L C_5  \norm{\Qvec}_{\mathcal{H}^2(\Omega)}.$ 
   
   This conclude the proof of Theorem \ref{thm: existance-uniqueness}.}


\section{Conclusions}

In this paper, we obtain \textit{a priori} error estimates for the conforming finite element approximation of local and global minimizers of the Landau-de Gennes  $Q$-tensor energy with either smooth or singular bulk potentials and anisotropic elastic energies. The analysis relies on the inf-sup stability of the linearized operator around a regular solution ($Q$) and around its conforming interpolant. This plays a pivotal role in the Newton-Kantorovich theorem. Existence and local uniqueness of discrete solutions and a priori error estimates follow.

For numerical implementation, the Newton scheme described in Theorem \ref{Kantorovich theorem} requires the computation of the second-order derivatives of the nonlinear potentials \eqref{LDG-bulk-density} and \eqref{BM-bulk-density} with respect to $Q$. The treatment of the LDG bulk potential $\psi_b^{LDG}$ in \eqref{LDG-bulk-density} is relatively straightforward and can be seen as a natural extension of the algorithm in \cite{DGFEM} to three dimensions. 

The primary challenge arises in handling the bulk potential $\psi_b^{BM}$ in \eqref{BM-bulk-density}, as the singular potential $f$ within this density is only implicitly defined as a function of $Q$. In \cite{Schimming2021}, the authors propose a numerical method to compute the Hessian of $f$ using a "dual minimization approach" that determines the optimal Lagrange multiplier for the constraint in $\mathcal{A}_\Qvec$ (see Section \ref{Problem formulation}), combined with Lebedev quadrature \cite{Lebedev1999} to compute the surface integrals over the unit sphere.


We here show that the crucial Hessian computation of $f$ from \cite{Schimming2021},  combined with a standard Newton method, allows to solve the Euler-Lagrange equations, as presented in Theorem \ref{Kantorovich theorem}, with guaranteed a priori error estimates.

		\bibliographystyle{amsplain}
		\bibliography{References}
  
\end{document}